\documentclass[11pt,a4paper,capitalise]{amsart}
\usepackage[utf8]{inputenc}
\usepackage{url,mathtools,xspace}
\usepackage{amsmath,amsthm,amssymb}
\usepackage[colorlinks=true]{hyperref}
\usepackage[capitalise]{cleveref}

\usepackage{subcaption}
\usepackage{fancyvrb}
\usepackage{tikz}
\usetikzlibrary{topaths}
\usepackage{pgfplots}
\usepgfplotslibrary{polar}
\pgfplotsset{compat=newest}
\tikzset{vertex/.style={circle, style=draw, inner sep=2pt, minimum size=8pt}}
\tikzset{vertex2/.style={circle, style=draw, inner sep=2pt, minimum size=8pt}}

\newcommand{\SageMath}{{\tt SageMath}}%
\newcommand{\FindStat}[1]{%
  \ifx&#1&%
  \url{www.findstat.org}
  \else \url{www.findstat.org/#1}
  \fi}%
\newcommand{\OEIS}[1]{\url{www.oeis.org/#1}}
\usepackage{biblatex}
\newtheorem{theorem}{Theorem}
\newtheorem{proposition}[theorem]{Proposition}
\newtheorem{corollary}[theorem]{Corollary}
\newtheorem{lemma}[theorem]{Lemma}

\theoremstyle{definition}

\newtheorem{example}[theorem]{Example}

\newcommand{\Dfn}[1]{\emph{\color{blue}#1}} 

\DeclareMathOperator{\cNhd}{\bar{N}}
\DeclareMathOperator{\Nhd}{N}
\DeclareMathOperator{\fix}{fix}
\newcommand{\sibling}{s}
\newcommand{\tuft}{t}
\newcommand{\EInProofOfThm}{\mathcal Y}
\newcommand{\comating}{\mathcal D}
\newcommand{\Tuft}{\mathcal T}
\newcommand{\Patch}{\mathcal P}

\newcommand{\decoratedgraphs}{connected $3$-sort graphs\xspace}
\newcommand{\decorated}{\mathcal G_{\!\neq 2}}
\newcommand{\decoratedgeneric}{\mathcal H_{\neq 2}}

\newcommand{\connected}{{\mathrm{c}}}
\newcommand{\Par}{\mathrm{Par}}
\newcommand{\Graphs}{\mathcal G}
\newcommand{\Sets}{E}

\newcommand{\dec}{{\operatorname{dec}}}

\title{Unexpectedly, a Symmetry on Unlabeled Graphs}

\author[Fürnsinn]{Florian Fürnsinn}
\author[Gangl]{Moritz Gangl}
\address[Fürnsinn, Gangl]{Fakult\"{a}t f\"{u}r Mathematik, Universität Wien, Vienna, Austria}
\email{florian.fuernsinn@univie.ac.at}
\email{moritz.gangl@univie.ac.at}
\author[Rubey]{Martin Rubey}
\address[Rubey]{Fakult\"{a}t f\"{u}r Mathematik und Geoinformation, TU Wien, Vienna, Austria}
\email{martin.rubey@tuwien.ac.at}

\begin{document}

\begin{abstract}
  We exhibit the joint symmetric distribution of the following two
  parameters on the set of unlabeled simple connected graphs
  with~$n$ vertices.  The first parameter is the maximal number of
  leaves attached to a vertex.  The second parameter is the size of
  the largest set of vertices sharing the same closed neighborhood
  minus~$1$.

  Apparently, this is the first example of a natural non-trivial
  equidistribution of graph parameters on unlabeled simple connected graphs
  with a fixed set of vertices.

  Our proof is enumerative, using the theory of species.  Exhibiting
  an explicit bijection interchanging the two parameters remains an
  open problem.
\end{abstract}

\maketitle

\section{Introduction}
Symmetries are an essential part of human life.  Some seem obvious to
us now, such as the rotational symmetry of a wheel, or the reflective
symmetry of a balance scale.  Other symmetries, such as the symmetry
between area and bounce on Dyck paths are still
mysterious~\cite[Open Problem~3.11]{MR2371044}.\footnote{Even in 2026, some people still ask why
  mirrors reflect horizontally but not vertically.}  In mathematical
terms, a symmetry is a bijection on a set of objects.  As the preceding example demonstrates, in some
cases, we know about the existence of a symmetry but are unable to
exhibit it explicitly.

There are many known bijections on sets
classically studied in combinatorics, such as permutations, integer
partitions or Dyck paths.  However, to the best of our
knowledge\footnote{\url{https://mathoverflow.net/questions/313532}},
there is not a single known \lq natural\rq\ bijection from the set of
all unlabeled simple connected graphs with a given finite
set of vertices to itself.  Enlarging the domain to all graphs we are
aware of taking the complement of a graph.  Alternatively, we can restrict the
domain to three-connected planar graphs and remove the constraint that the number
of vertices is preserved.
Then a celebrated result of Whitney shows that mapping a graph to its planar dual is an involution~\cite[Thm.~11]{Whi32}. Of
course, it is possible to modify these maps to obtain involutions
preserving connectivity, but such a modification does not seem to be a very natural thing to do.

In fact\footnote{\url{https://mathoverflow.net/questions/312823}},
we do not even know of any pair of \lq natural\rq\ equidistributed
parameters on unlabeled simple connected graphs appearing in the literature.  This may be surprising, since many graph parameters have been considered and studied, albeit mostly in terms of inequalities satisfied among them.  The \FindStat{} database, as of~2025, lists about
300~such statistics, without any conjectural equidistributions.

Note that any equidistribution implies the existence of a bijection that maps one of the parameters to the other.
Thus, our aim is to present a pair of non-trivial graph parameters
which have the same distribution on the set of unlabeled
simple connected graphs with a fixed number of vertices.

We consider only finite simple graphs.  A graph $G$ consists of a set of vertices $V(G)$ and a set of edges $E(G)$.  Given a
vertex $v$, the (open) \Dfn{neighborhood} of $v$ is the set
$
\Nhd(v)\coloneqq\{u\in V(G)\mid u\text{ adjacent to } v\},
$
whereas the set
$
\cNhd(v)\coloneqq\{v\} \cup \Nhd(v)
$
is the \Dfn{closed neighborhood} of $v$.
Two vertices are \Dfn{siblings} if they share the same
closed neighborhood.  Finally, a \Dfn{leaf} or \Dfn{endpoint} is a
vertex of degree one.

We consider the following two parameters on graphs.  The \Dfn{sibling
  number}
\[
  \sibling(G) %
  \coloneqq \max_{v\in V(G)}\left\lvert\{u\in V(G)\setminus\{v\}\mid\cNhd(u)=\cNhd(v)\}\right\rvert
\]
is the maximal number of other vertices sharing the
same \emph{closed} neighborhood with any given vertex.  The \Dfn{tuft
  number\/\footnote{\FindStat{St001826}}} $\tuft(G)$ is the maximal
number of leaves adjacent to the same vertex.

An unlabeled graph is an isomorphism class of graphs.
\Cref{fig:Graphs4} displays all unlabeled connected graphs with four vertices, together with their sibling and tuft numbers. Among all graphs with $n>2$ vertices, the graph with maximal
sibling number is the complete graph $K_n$, which has sibling number
$n-1$ and tuft number~$0$.  On the other hand, the unique graph
maximizing the tuft number is the star graph $S_n$, which has $n-1$
leaves adjacent to a single vertex, and no siblings.
\begin{figure}
    \centering
    \begin{tikzpicture}
        \node[vertex] (V1) at (0,0) {};
        \node[vertex] (V2) at (1,0) {};
        \node[vertex] (V3) at (1,1) {};
        \node[vertex] (V4) at (0,1) {};
        \node (A) at (.5, -1) {$\sibling(G)=0$};
        \node (B) at (.5, -1.5) {$\tuft(G)=3$};
        \draw (V2)--(V1)--(V3);
        \draw (V1)--(V4);
    \end{tikzpicture}
    \begin{tikzpicture}
        \node[vertex] (V1) at (0,0) {};
        \node[vertex] (V2) at (1,0) {};
        \node[vertex] (V3) at (1,1) {};
        \node[vertex] (V4) at (0,1) {};
        \node (A) at (.5, -1) {$\sibling(G)=0$};
        \node (B) at (.5, -1.5) {$\tuft(G)=1$};
        \draw (V4)--(V1)--(V2)--(V3);
    \end{tikzpicture}
    \begin{tikzpicture}
        \node[vertex] (V1) at (0,0) {};
        \node[vertex] (V2) at (1,0) {};
        \node[vertex] (V3) at (1,1) {};
        \node[vertex] (V4) at (0,1) {};
        \node (A) at (.5, -1) {$\sibling(G)=0$};
        \node (B) at (.5, -1.5) {$\tuft(G)=0$};
        \draw (V1)--(V2)--(V3)--(V4)--(V1);
    \end{tikzpicture}
    \begin{tikzpicture}
        \node[vertex] (V1) at (0,0) {};
        \node[vertex] (V2) at (1,0) {};
        \node[vertex] (V3) at (1,1) {};
        \node[vertex] (V4) at (0,1) {};
        \node (A) at (.5, -1) {$\sibling(G)=1$};
        \node (B) at (.5, -1.5) {$\tuft(G)=1$};
        \draw (V1)--(V2)--(V3)--(V1)--(V4);
    \end{tikzpicture}
    \begin{tikzpicture}
        \node[vertex] (V1) at (0,0) {};
        \node[vertex] (V2) at (1,0) {};
        \node[vertex] (V3) at (1,1) {};
        \node[vertex] (V4) at (0,1) {};
        \node (A) at (.5, -1) {$\sibling(G)=1$};
        \node (B) at (.5, -1.5) {$\tuft(G)=0$};
        \draw (V1)--(V2)--(V3)--(V1)--(V4)--(V3);
    \end{tikzpicture}
    \begin{tikzpicture}
        \node[vertex] (V1) at (0,0) {};
        \node[vertex] (V2) at (1,0) {};
        \node[vertex] (V3) at (1,1) {};
        \node[vertex] (V4) at (0,1) {};
        \node (A) at (.5, -1) {$\sibling(G)=3$};
        \node (B) at (.5, -1.5) {$\tuft(G)=0$};
        \draw (V1)--(V2)--(V3)--(V1)--(V4)--(V3);
        \draw (V2)--(V4);
    \end{tikzpicture}
    \caption{Sibling and tuft numbers of all unlabeled connected graphs with $4$ vertices.}
    \label{fig:Graphs4}
\end{figure}

Our main theorem, initially conjectured by Per Alexandersson and
the third author, exhibits a non-trivial equidistribution among graph parameters.  
\begin{theorem}\label{thm:main}
  The sibling number and the tuft number have joint symmetric
  distribution on the set of unlabeled connected graphs with~$n$
  vertices.  Put differently, the isomorphism type generating series
  $\sum_{G} x^{\sibling(G)} y^{\tuft(G)}$, where the sum is taken over all unlabeled connected graphs
  with~$n$ vertices, is symmetric in~$x$ and $y$.
\end{theorem}
The sibling number and tuft number of a graph are the maximum sibling number and tuft number, respectively, of its connected components.  Therefore, equidistribution of the parameters on connected graphs is a stronger property than equidistribution on the set of all graphs.  The maximal number of vertices sharing the same open neighborhood and the tuft number do not have joint symmetric distribution.

To illustrate \Cref{thm:main}, the graphs with six vertices with sibling number~$1$ and
tuft number~$2$ and those with sibling number~$2$ and tuft
number~$1$ are displayed in \Cref{fig:Graphs6}.
\begin{figure}[h]
            \centering
            \begin{tikzpicture}
                \node[vertex] (V1) at (0,0) {};
                \node[vertex] (V2) at (-0.5,-0.866) {};
                \node[vertex] (V3) at (0.5,-0.866) {};
                \node[vertex] (V4) at (0,-1.732) {};
                \node[vertex] (V5) at (-0.5,-2.598) {};
                \node[vertex] (V6) at (0.5,-2.598) {};
                \draw (V1)--(V2)--(V4)--(V5);
                \draw (V1)--(V3)--(V4)--(V6);
                \draw (V2)--(V3);
                \node (A) at (0, -3.598) {$\sibling(G)=1$};
                \node (B) at (0, -4.098) {$\tuft(G)=2$};
            \end{tikzpicture}
            \hspace{1cm}
            \begin{tikzpicture}
                \node[vertex] (V1) at (-0.5,0) {};
                \node[vertex] (V2) at (0.5,0) {};
                \node[vertex] (V3) at (0,-0.866) {};
                \node[vertex] (V4) at (0,-1.732) {};
                \node[vertex] (V5) at (-0.5,-2.598) {};
                \node[vertex] (V6) at (0.5,-2.598) {};
                \draw (V3)--(V1)--(V2)--(V3)--(V4)--(V5);
                \draw (V4)--(V6);
                \node (A) at (0, -3.598) {$\sibling(G)=1$};
                \node (B) at (0, -4.098) {$\tuft(G)=2$};
            \end{tikzpicture}
            \hspace{2cm}
            \begin{tikzpicture}
                \node[vertex] (V1) at (0,0) {};
                \node[vertex] (V2) at (0,-0.866) {};
                \node[vertex] (V3) at (-0.5,-1.732) {};
                \node[vertex] (V4) at (0.5,-1.732) {};
                \node[vertex] (V5) at (0,-2.598) {};
                \node[vertex] (V6) at (0,-3.464) {};
                \draw (V1)--(V2)--(V3)--(V4)--(V2)--(V5);
                \draw (V5)--(V3)--(V1)--(V4)--(V5)--(V6);
                \node (A) at (0, -4.464) {$\sibling(G)=2$};
                \node (B) at (0, -4.964) {$\tuft(G)=1$};
            \end{tikzpicture}
            \hspace{1cm}
            \begin{tikzpicture}
                \node[vertex] (V1) at (0,0) {};
                \node[vertex] (V2) at (0.5,-0.866) {};
                \node[vertex] (V3) at (-0.5,-0.866) {};
                \node[vertex] (V4) at (0,-1.732) {};
                \node[vertex] (V5) at (0,-2.598) {};
                \node[vertex] (V6) at (0,-3.464) {};
                \draw (V1)--(V2)--(V3)--(V1)--(V4)--(V5)--(V6);
                \draw (V2)--(V4)--(V3);
                \node (A) at (0, -4.464) {$\sibling(G)=2$};
                \node (B) at (0, -4.964) {$\tuft(G)=1$};
            \end{tikzpicture}
            \caption{The two unlabeled connected graphs with~$6$ vertices with sibling number~$1$
              and tuft number~$2$, and the two unlabeled connected graphs with~$6$ vertices with sibling number~$2$ and tuft number~$1$.}
            \label{fig:Graphs6}
\end{figure}

Graphs without siblings and graphs without vertices sharing the same open neighborhood were studied previously by several authors.  Apparently, Sumner~\cite{zbMATH03417503} was the first to formally introduce terminology.  He calls a graph \Dfn{point determining} if all of its vertices have distinct open neighborhoods, and \Dfn{point distinguishing} if all of its vertices have distinct closed neighborhoods.  He also notes that a graph is point distinguishing if and only if its complement is point determining.

Bull and Pease~\cite{bullpease1989} re-introduced point determining graphs as \lq mating graphs\rq, to describe possible mating types of biological species.  They provide the number of unlabeled point determining graphs with at most ten vertices, and ask for a general formula.  Soon after, such a formula was supplied by Read~\cite{read1989enumeration}, using the cycle index of graphs.

Jovović\footnote{cf.~comments in \OEIS{A004108}} then observed that the number of unlabeled connected point determining graphs with $n$ vertices coincides with the number of unlabeled connected leafless graphs with $n$ vertices, if $n\neq 2$.  This was shown bijectively by Kilibarda~\cite{MR2300242}.  A proof using combinatorial species building on Read's idea was then provided by Gessel and Li~\cite{MR2739506}.  In fact, they show the following.
\begin{theorem}[Kilibarda 2007, Gessel and Li 2011]\label{thm:KilGL}
    The number of unlabeled connected point determining graphs with $n$ vertices, the number of unlabeled connected point distinguishing graphs with $n$ vertices, and the number of unlabeled connected leafless graphs with $n$ vertices, coincide for $n\neq 2$.
\end{theorem}

Gessel and Li used the language of combinatorial species
that was introduced by Joyal~\cite{MR0927763,MR0633783}, and further
developed by Bergeron, Labelle and Leroux~\cite{MR1629341}, to provide
an alternative proof of Kilibarda's result. We follow Gessel and Li's approach to prove \Cref{thm:main}. More precisely, we
express the species of graphs with sibling number at
most~$\sibling$ and tuft number at most~$\tuft$ as a
composition of the species of graphs without siblings and a virtual
species.  It then turns out that this virtual species has an
isomorphism type generating series which is obviously symmetric
in~$\sibling$ and~$\tuft$.

Since our main motivation was to exhibit a symmetry on graphs, let us mention some obstructions we encountered trying to find an explicit bijection.  First of all, the sibling and the tuft number are not equidistributed on the set of (labeled) graphs.  This fact also makes combinatorial species a natural tool to approach the problem.  Although leaves cannot be adjacent to vertices having siblings in graphs other than $K_2$, it turns out that it is impossible to replace each group of siblings with a vertex having the appropriate number of leaves and vice versa.  The reason is that the number of groups of siblings of maximal cardinality and the number of vertices adjacent to the maximal number of leaves are not equidistributed.  In \Cref{sec:bijection} we describe some tentative first steps towards a bijection.

However, we are able to prove the following refinement of \Cref{thm:main}.  Let the
\Dfn{reduction} of a graph -- other than $K_2$ -- be the graph
obtained by repeatedly removing all leaves and contracting each group
of siblings to a single vertex.  Trivially, the reduction is a
leafless graph without siblings.  \Cref{fig:example reducing}
provides an example of the reduction process.
\begin{figure}[ht]
\centering
\begin{tikzpicture}[baseline={(0,0)}]
     \node[vertex2] (V10) at (-0.525,0.05) {};
     \node[vertex2] (V11) at (0.05,-0.525) {};
     \node[vertex2] (V12) at (0.26,0.26) {};

      \node[vertex2] (V21) at (2,0) {};
      \node[vertex2] (V22) at (0,2) {};

      \node[vertex2] (V3) at (1,-1) {};
      \node[vertex2] (V4) at (-1,1) {};

      \node[vertex2] (V5) at (-1,-1) {};
      \node[vertex2] (V51) at (-1.8,-.2) {};
      \node[vertex2] (V52) at (-1.8,-1) {};
      \node[vertex2] (V53) at (-1.8,-1.8) {};

      \node[vertex2] (V6) at (2,2) {};

      \node[vertex2] (V71) at (-.4,-2) {};
      \node[vertex2] (V712) at (-.8,-2.8) {};
      \node[vertex2] (V713) at (0,-2.8) {};

      \node[vertex2] (V72) at (.4,-2) {};

      \node[vertex2] (V81) at (-2.4,1.6) {};
      \node[vertex2] (V811) at (-2.4,.8) {};
      \node[vertex2] (V82) at (-1.6,2.4) {};

      \draw (V10)--(V11)--(V12)--(V10) (V10)--(V21)--(V3) (V10)--(V22)--(V3) (V4)--(V5)--(V10) (V11)--(V21) (V21)--(V22) (V11)--(V22) (V5)--(V11) (V12)--(V21) (V12)--(V22) (V5)--(V12) (V3)--(V5) (V21)--(V4) (V21)--(V6) (V22)--(V4) (V22)--(V6) (V3)--(V71) (V4)--(V81) (V4)--(V82) (V81)--(V82) (V71)--(V5) (V3)--(V72)--(V71) (V72)--(V5) (V51)--(V5)--(V52) (V71)--(V712) (V712)--(V713)--(V71) (V811)--(V81) (V53)--(V5);
  \end{tikzpicture}\hspace{.5cm}
     \begin{tikzpicture}[baseline={(0,0)}]
     \node[vertex2] (V1) at (-0.07,-0.07) {};
      \node[vertex2] (V2) at (1,1) {};
      \node[vertex2] (V3) at (1,-1) {};
      \node[vertex2] (V4) at (-1,1) {};
      \node[vertex2] (V5) at (-1,-1) {};

      \node[vertex2] (V6) at (2,2) {};
      \node[vertex2] (V72) at (.4,-2) {};
      \node[vertex2] (V71) at (-.4,-2) {};
      \node[vertex2] (V711) at (-.4,-2.8) {};
      \node[vertex2] (V8) at (-2,2) {};

      \draw (V1)--(V2)--(V3) (V4)--(V5)--(V1) (V3)--(V5) (V2)--(V4) (V2)--(V6) (V3)--(V71)--(V72)--(V3) (V4)--(V8) (V711)--(V71)--(V5)--(V72);
  \end{tikzpicture}\hspace{.5cm}
  \begin{tikzpicture}[baseline={(0,0)}]
     \node[vertex2] (V1) at (-0.07,-0.07) {};
      \node[vertex2] (V2) at (1,1) {};
      \node[vertex2] (V3) at (1,-1) {};
      \node[vertex2] (V4) at (-1,1) {};
      \node[vertex2] (V5) at (-1,-1) {};
      \node[vertex2] (V7) at (0,-2) {};

      \draw (V1)--(V2)--(V3) (V4)--(V5)--(V1) (V3)--(V5) (V2)--(V4) (V3)--(V7) (V7)--(V5);
  \end{tikzpicture}
    \caption{The graph on the left is reduced in two steps, each removing leaves and then contracting siblings, to the leafless point distinguishing graph depicted on the right.}
    \centering
    \label{fig:example reducing}
    \end{figure}

\begin{theorem}\label{thm:R}
  Let $R$ be a connected leafless point distinguishing graph.  The
  sibling number and the tuft number have joint symmetric distribution on
  the set of unlabeled connected graphs with $n$ vertices that
  reduce to~$R$.
\end{theorem}
The reduction operation has some attractive properties.  For example, we show in \Cref{thm:inducedcycles} that a graph $G$ contains an induced $k$-cycle if and only if the reduction of $G$ contains an induced $k$-cycle, for $k\geq 4$.  Together with the previous theorem, this implies that the sibling and the tuft number also have joint symmetric distribution on the set of connected chordal graphs, that is, graphs which do not contain an induced cycle with $4$ or more vertices.

In a different direction, we show in \Cref{prop:comating} that the species of point distinguishing graphs reducing to a given leafless point distinguishing graph $R$ can be expressed in terms of the species of point distinguishing graphs reducing to the single vertex graph.  In particular, this implies that the enumeration of graphs with restricted sibling and restricted tuft number reducing to a graph $R$ can be reduced to the case where $R$ is the single vertex graph.

At the heart of our proof of \Cref{thm:main} is a functional equation, expressing the species of graphs in terms of the species of point distinguishing graphs. As a further application, we exhibit another enumerative coincidence.  A graph is \Dfn{bi-point determining} if it is  simultaneously point determining and point distinguishing.  Equivalently, its automorphism group contains no transposition.  Combining an expression for the species of bi-point determining graphs due to Gessel and Li~\cite[Cor. 4.7]{MR2739506} with our functional equation, we find the following.
\begin{theorem}\label{cor:bipointdetermining}
  The number of unlabeled connected bi-point determining graphs with $n$ vertices equals the number of unlabeled connected leafless point distinguishing graphs with $n$ vertices.
\end{theorem}
We do not have a bijective proof for this theorem, either.

\medskip

{\bf Outline.} In \Cref{sec:species} we recall the main definitions
and some results from the theory of combinatorial species, mostly to
fix notation.  In particular, in \Cref{exam: comating graph species,exam: 3-labeled
  graphs,exam: flowers,exam: cultivations} we introduce the species
which appear in the following sections.  In \Cref{sec:patch decomp}, we provide a decomposition of the species of graphs, essential for our approach, generalizing the functional equation used by Gessel and Li.  As a first application, we prove \Cref{cor:bipointdetermining} in \Cref{sec:bipoint}.  We use the decomposition again in \Cref{sec:symmetry} for the
proof of \Cref{thm:main}.  In the
following two sections, we refine the decomposition to obtain
\Cref{thm:R}.  First, in \Cref{sec:reduction}, we show that removing
leaves and contracting siblings in any order will produce
the reduction of the graph.  Then, in \Cref{sec:symmetry-refined}, we
deduce \Cref{thm:R}.  In \Cref{sec:graph-classes}, we show that
induced cycles other than triangles are preserved by the reduction
process.  In particular, this implies that the sibling and the tuft
number also have joint symmetric distribution on the set of
unlabeled connected chordal graphs with $n$ vertices.  In
\Cref{sec:Description of Comating}, we provide a somewhat more
explicit description of the species occurring in the decompositions
used to prove our main theorems.  Finally, in \Cref{sec:bijection}, we collect some observations that may help finding a
bijective proof of \Cref{thm:R}.

\section{Combinatorial Species}
\label{sec:species}
Combinatorial species were introduced by Joyal as a formal framework
for enumeration~\cite{MR0927763,MR0633783}.  In this section, apart
from briefly recalling the definitions, we also introduce the species
appearing subsequently.  For any details, we recommend Joyal's
original articles and the textbook by Bergeron, Labelle and
Leroux~\cite{MR1629341}.

A \Dfn{combinatorial species} is a functor from the category of
finite sets with bijections to itself. Given a (finite) set of
\Dfn{labels} $U$, a species $F$ produces a (finite) set $F[U]$ of
\Dfn{$F$-structures}.  Moreover, given a bijection
$\sigma: U\to V$ it produces a bijection $F[\sigma]: F[U]\to F[V]$,
called the \Dfn{relabeling} or \Dfn{transport} of $F$-structures
along $\sigma$.  Often, the precise definition of relabeling for a
particular species is obvious from context and therefore omitted.  We
regard two species $F$ and $G$ as \Dfn{(combinatorially) equal\/\footnote{cf.~\cite[p.~21]{MR1629341}}} if there is a natural
isomorphism between $F$ and $G$ in the categorical sense.

The symmetric group $\mathfrak S_n$ acts on the set
$F[n] \coloneqq F[\{1, 2, \dots, n\}]$ by relabeling.
The $\mathfrak S_n$-orbits under this action are called \Dfn{isomorphism types of $F[n]$} or
\Dfn{unlabeled $F$-structures} of order (or size) $n$.  Accordingly,
two $F$-structures $a\in F[U]$ and $b\in F[V]$ are \Dfn{isomorphic}
if there is a bijection $\sigma:U\to V$ with $F[\sigma](a) = b$.

\begin{example}
The species of \Dfn{(simple) graphs} is defined as follows.  $\Graphs[U]$ is the set of simple graphs with vertex set $U$.  Unlabeled graphs are precisely the isomorphism types of $\Graphs$.
\end{example}

The \Dfn{exponential generating series} and the \Dfn{isomorphism
  type generating series} associated with a species $F$ are
\begin{align*}
  F(x) &\coloneqq \sum_{n\geq 0} \big\lvert F[n]\big\rvert \frac{x^n}{n!}\quad\text{and}\\
  \widetilde F(x) &\coloneqq \sum_{n\geq 0} \big\lvert\big\{\text{isomorphism types of $F[n]$}\big\}\big\rvert\; x^n,\\
\intertext{respectively.  Both are specializations of the \Dfn{cycle index series}}
  Z_F &\coloneqq \sum_{n\geq 0}\frac{1}{n!}%
  \sum_{\sigma\in\mathfrak S_n}\fix F[\sigma] p_1^{\sigma_1} p_2^{\sigma_2} \dots,
\end{align*}
where $\fix F[\sigma]$ denotes the number of fixed points of
$F[\sigma]$, $p_k$ is the $k$-th power sum symmetric function, and
$\sigma_k$ is the number of $k$-cycles of $\sigma$.
\begin{proposition}
  The exponential generating series of a species is obtained from its
  cycle index series by specializing $p_1=x$ and $p_k=0$ for $k>1$.

  The isomorphism type generating series of a species is obtained
  from its cycle index series by specializing $p_k=x^k$ for all
  $k\geq 1$.
\end{proposition}

A species $F$ is a \Dfn{subspecies} of a species $G$, if, for any set
of labels, every $F$-structure is also a $G$-structure, and the
relabeling bijections coincide.  We denote the \Dfn{restriction of a
  species}~$F$ to sets of cardinality~$m$ by $F_m$.  The
restrictions to sets of cardinality other than $m$, at least $m$, at most~$n$ and
sets of cardinalities from~$m$ to~$n$ are denoted by
$F_{\neq m},$ $F_{\geq m}$, $F_{\leq n}$ and $F_{m\leq\bullet \leq n}$, respectively.

Perhaps the most fundamental species is the \Dfn{species of sets}
$\Sets$, which produces a single structure for any set of
labels.  The species $\Sets_0$ is usually denoted by $1$, and the
species $\Sets_1$ is usually denoted by $X$ and referred to as the
\Dfn{singleton species}.  We have $\Sets(x) = \exp x$,
$\widetilde{\Sets}(x) = \frac{1}{1-x}$, and
\[
  Z_{\Sets}=\exp\left(\sum_{k>0}\frac{p_k}{k}\right).
\]
The zero species, producing no structure for any set of labels, is denoted by~$0$.
\begin{example}\label{exam: comating graph species}
Recall that a \Dfn{point distinguishing graph} is a graph without siblings.
We define $\comating^\connected[U]$ as the set of connected point distinguishing graphs with vertex
set $U$.  For a given connected point distinguishing graph $D\in\comating^\connected[U]$ and a bijection
$\sigma: U\to V$, the graph $\comating^\connected[\sigma](D)$ is obtained from
$D$ by replacing the vertex labeled $u$ with a vertex labeled
$\sigma(u)$.  Thus, the species of connected point distinguishing graphs $\comating^\connected$ is a subspecies of the species
$\Graphs^\connected$ of connected graphs.

 A \Dfn{point determining graph} is a graph without vertices sharing the same open neighborhood.  A \Dfn{bi-point determining graph} is a graph which is point distinguishing and point determining.  Thus, the species of bi-point determining graphs $\mathcal B^\connected$ is the subspecies of $\comating^\connected$ whose structures are also point determining.
\end{example}

A \Dfn{multisort species of $k$ sorts} is a functor from the
category whose objects are $k$-tuples of finite sets (short: \lq $k$-sets\rq) and whose morphisms are $k$-tuples of bijections, to the category of finite sets with bijections.
Given a $k$-set $(U_1,\ldots, U_k)$, the elements of each set
$U_i$ are referred to as the \Dfn{labels of sort~$i$}.  The concepts introduced for $1$-sort species carry over to the multisort context in the obvious way.

We denote by $X$, $Y,$ and $Z$ the singleton species of
three different sorts.  That is, given a $3$-set $U=(U_1, U_2, U_3)$,
the set $X[U]$ contains a single structure if $U_1$ has one element
and $U_2$ and $U_3$ are empty.  Otherwise, $X[U]$ is empty. For convenience,
we refer to labels of sort $1$ as labels of sort $X$, etc. To emphasize that a species $F$ has, for example, $3$ sorts, we sometimes write $F(X,Y,Z)$. This notation is compatible with the composition of species defined below. However, we alert the reader that by abuse of notation $X$ denotes both a $1$-sort singleton species and a multisort singleton species.

In the following definitions, the addition of $k$-sets always refers
to their component-wise disjoint union.

\begin{example}\label{exam: 3-labeled graphs}
  The species $\decorated^\connected(X,Y,Z)$ of \Dfn{\decoratedgraphs}, with leaves of sort $Y$ and siblings of sort $Z$, is defined as follows.
  Given a $3$-set $U=(U_1, U_2, U_3)$, let $\decorated^\connected[U]$
  be the set of connected graphs with vertex set
  $U_1 + U_2 + U_3$, such that the set of leaves is $U_2$, the set of vertices having siblings is $U_3$, and $U_1$ is the set of remaining vertices.
  
  The complete graph with two vertices, $K_2$, is the unique
  connected graph that has a leaf which is simultaneously a sibling.
  Thus, there is no $\decorated^\connected$-structure
  isomorphic to $K_2$.  

  The isomorphism types of $\decorated^\connected(X, Y, Z)$ with at
  most three vertices and the corresponding isomorphism types of
  $\Graphs^\connected(X)$ are depicted in \Cref{fig:Decorations}.

  Note that $\decorated^\connected(X, X, X) = \Graphs_{\neq 2}^\connected(X)$, and $\decorated^\connected(X, X, 0)$ is the species of connected point distinguishing graphs
  $\comating^\connected,$ whereas $\decorated^\connected(X, 0, 0)$ is the species of
  connected leafless point distinguishing graphs.
\end{example}

\begin{figure}[h]
  \[
  \begin{array}{c|c|c|c|c|c}
    & 0 & 1 & 2 & 3 & \cdots \\\hline
    \begin{tikzpicture}[baseline={(current bounding box.center)}]
      \node () at (0,0) {$\Graphs^\connected(X)$};
    \end{tikzpicture} & 
    \begin{tikzpicture}[baseline={(current bounding box.center)}]
      \node () at (0,0) {$\emptyset$};
    \end{tikzpicture} & 
    \begin{tikzpicture}[baseline={(current bounding box.center)}]
      \node[vertex] () at (0,0) {};
    \end{tikzpicture} & 
    \begin{tikzpicture}[baseline={(current bounding box.center)}]
      \node[vertex] (V1) at (0,0) {};
      \node[vertex] (V2) at (1,0) {};
      \draw (V1) -- (V2);
    \end{tikzpicture} & 
    \begin{tikzpicture}[baseline={(current bounding box.center)}]
      \node () at (0.866,-.5) {\phantom{$Y$}};
      \node () at (0.866,.5) {\phantom{$Y$}};
      \node () at (0,0) {\phantom{$X$}};
      \node[vertex] (V1) at (0.866,-.5) {};
      \node[vertex] (V3) at (0.866,.5) {};
      \node[vertex] (V2) at (0,0) {};
      \draw (V1) -- (V2) -- (V3);
    \end{tikzpicture}
    \begin{tikzpicture}[baseline={(current bounding box.center)}]
      \node () at (0.866,-.5) {\phantom{$Z$}};
      \node () at (0.866,.5) {\phantom{$Z$}};
      \node () at (0,0) {\phantom{$Z$}};
      \node[vertex] (V1) at (0.866,-.5) {};
      \node[vertex] (V3) at (0.866,.5) {};
      \node[vertex] (V2) at (0,0) {};
      \draw (V1) -- (V2) -- (V3) -- (V1);
    \end{tikzpicture} & 
    \cdots \\\hline
    \begin{tikzpicture}[baseline={(current bounding box.center)}]
      \node () at (0,0) {$\decorated^\connected(X, Y, Z)$};
    \end{tikzpicture} & 
    \begin{tikzpicture}[baseline={(current bounding box.center)}]
      \node () at (0,0) {$\emptyset$};
    \end{tikzpicture} & 
    \begin{tikzpicture}[baseline={(current bounding box.center)}]
      \node () at (0,0) {$X$};
    \end{tikzpicture} & 
    \begin{tikzpicture}[baseline={(current bounding box.center)}]
      \node () at (0,0) {$\emptyset$};
    \end{tikzpicture} & 
    \begin{tikzpicture}[baseline={(current bounding box.center)}]
      \node (V1) at (0.866,-.5) {$Y$};
      \node (V3) at (0.866,.5) {$Y$};
      \node (V2) at (0,0) {$X$};
      \draw (V1) -- (V2) -- (V3);
    \end{tikzpicture}
    \begin{tikzpicture}[baseline={(current bounding box.center)}]
      \node (V1) at (0.866,-.5) {$Z$};
      \node (V3) at (0.866,.5) {$Z$};
      \node (V2) at (0,0) {$Z$};
      \draw (V1) -- (V2) -- (V3) -- (V1);
    \end{tikzpicture} & 
    \cdots
  \end{array}
  \]
  \caption{The first few isomorphism types of the species of connected graphs
    $\Graphs^\connected(X)$ and the species of \decoratedgraphs
    $\decorated^\connected(X, Y, Z)$.  In the multisort case, we indicate the sorts in illustrations by the corresponding letters.}
  \label{fig:Decorations}
\end{figure}

We will make use of the sum, the product and the composition of two
species.  The set of structures of the \Dfn{sum} of two species $F$
and $G$ is the disjoint union of the set of $F$-structures and the
set of $G$-structures, \[(F+G)[U] \coloneqq F[U]+G[U].\]  The set of structures
of the \Dfn{product} of $F$ and $G$ is
\[
  (F\cdot G)[U] \coloneqq \sum_{U = V + W} F[V] \times G[W],
\]
where the sum is over all ordered set partitions of $U$ into two
(possibly empty) parts.  Finally, let $F$ be a $1$-sort species and
let $G$ be a $k$-sort species whose restriction to the empty set is
zero.  Then the set of structures on the $k$-set $U$ of the
\Dfn{(partitional) composition} of $F$ and $G$ is
\[
  (F\circ G)[U] \coloneqq \sum_{\pi\in\Par[U]} F[\pi]\times\prod_{B\in\pi} G[B],
\]
where the sum is over all set partitions of $U$ into non-empty
blocks.  In other words, a structure in $\Par[(U_1,\dots,U_k)]$ is a
set of $k$-tuples, each of which has at least one non-empty
component, whose component-wise union is $U$.  The composition
$F(G_1,\dots,G_m)$ of an $m$-sort species $F$ with $m$ species
$G_1,\dots,G_m$ is defined similarly.

\begin{example} \label{ex: productofsets}
    Let $F$ and $G$ be species. Then $\Sets(F)\cdot\Sets(G)=\Sets(F+G)$.
\end{example}

\begin{example}\label{exam: flowers}
    A \Dfn{tuft} is a structure consisting of a designated \Dfn{root} of
sort $X$ and a non-empty set of leaves of sort $Y$.  Thus, the species
of tufts equals
\[
  \Tuft(X, Y) \coloneqq X\Sets_{\geq 1}(Y).
\]
\end{example}
\begin{figure}[h]
    \centering
    \begin{minipage}{.495\textwidth}
        \centering
        \resizebox{.99\textwidth}{!}{
        \begin{tikzpicture}[line width=.4mm]
            \node at (0.7,0)
            {\includegraphics[width=1.2\linewidth]{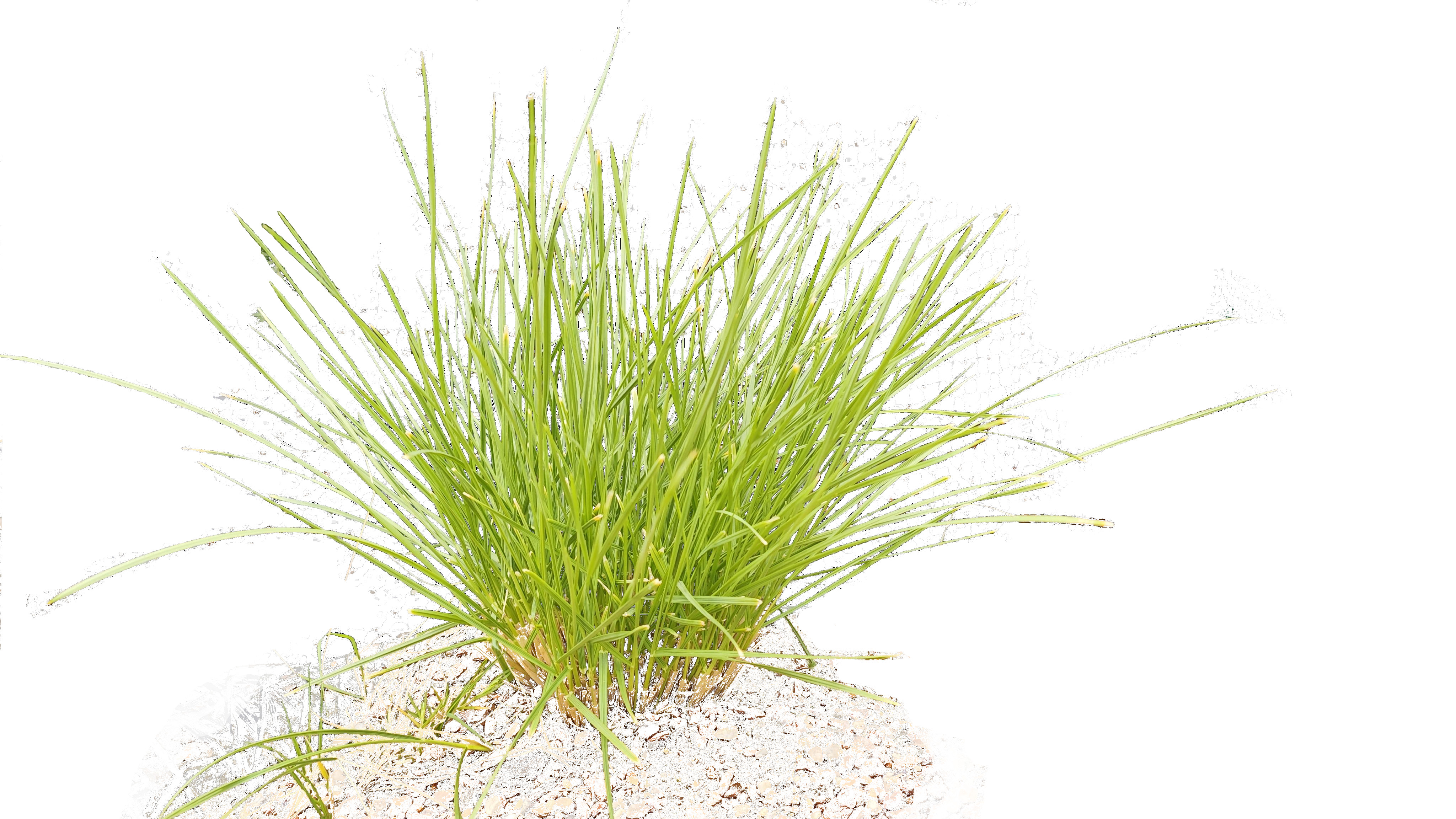}};
            \node (X) at (0,-1.7) {\Large $X$};
            \node (Y1) at (0,1.4) {\Large $Y$};
            \node (Y2) at (-1.8,.3) {\Large $Y$};
            \node (Y3) at (-1,1.2) {\Large $Y$};
            \node (Y4) at (1,1.3) {\Large $Y$};
            \node (Y5) at (1.8,.4) {\Large $Y$};
            \node (Y6) at (2,-.2) {\Large $Y$};
            \draw (Y1) -- (X) -- (Y2);
            \draw (Y3) -- (X) -- (Y4);
            \draw (Y5) -- (X) -- (Y6);
        \end{tikzpicture}}
        \captionof{figure}{A tuft}
        \label{fig:tuft}
    \end{minipage}
    \begin{minipage}{.495\textwidth}
        \centering
        \resizebox{.99\textwidth}{!}{
        \begin{tikzpicture}[line width=.4mm]
            \node at (0.7,0)
            {\includegraphics[width=0.6\linewidth]{tuft.jpg}};
            \node (X) at (0.35,-.85) {\large $X$};
            \node (Y1) at (0.35,.7) {\large $Y$};
            \node (Y2) at (-0.55,.15) {\large $Y$};
            \node (Y3) at (-0.15,.6) {\large $Y$};
            \node (Y4) at (0.85,.65) {\large $Y$};
            \node (Y5) at (1.25,.2) {\large $Y$};
            \node (Y6) at (1.35,-.3) {\large $Y$};
            \draw (Y1) -- (X) -- (Y2);
            \draw (Y3) -- (X) -- (Y4);
            \draw (Y5) -- (X) -- (Y6);
            \node at (4.7,0)
            {\includegraphics[width=0.6\linewidth]{tuft.jpg}};
            \node (Xa) at (4.35,-.85) {\large $X$};
            \node (Y2a) at (3.45,.15) {\large $Y$};
            \node (Y5a) at (5.25,.2) {\large $Y$};
            \draw (Y2a) -- (Xa) -- (Y5a);
            \node at (0.3,-3)
            {\includegraphics[width=0.3\linewidth]{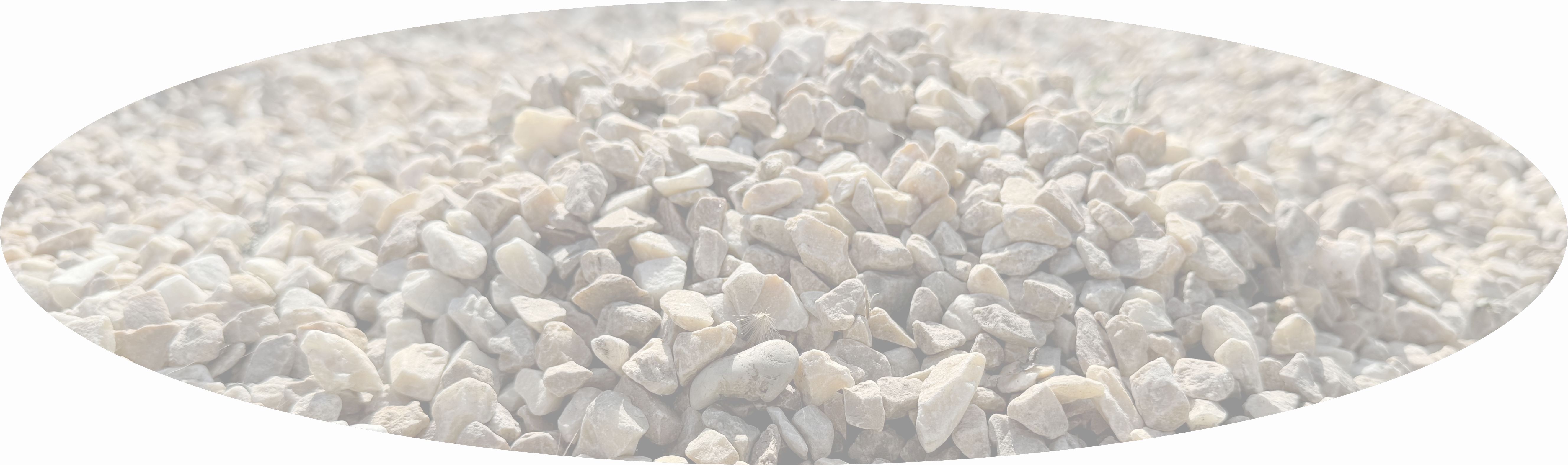}};
            \node (Za) at (0.35,-2.85) {\large $Z$};
            \node at (2.3,-2)
            {\includegraphics[width=0.3\linewidth]{root4.jpg}};
            \node (Zb) at (2.35,-1.85) {\large $Z$};
            \node at (4.3,-3)
            {\includegraphics[width=0.3\linewidth]{root4.jpg}};
            \node (Zc) at (4.35,-2.85) {\large $Z$};
        \end{tikzpicture}}
        \captionof{figure}{A patch}
        \label{fig:patch}
    \end{minipage}
\end{figure}
\begin{example}\label{exam: cultivations}
A \Dfn{patch} is a non-empty structure consisting of a
possibly empty set of tufts together with a possibly empty set of
additional leafless \Dfn{roots}, subject to the following rule: if there
is precisely one leafless root, it is of sort $X$, otherwise the leafless roots
are of sort $Z$. In particular, the roots of a patch are exactly the vertices of sort $X$ and $Z$.  Thus, the species of patches equals
\[
  \Patch(X, Y, Z) %
  \coloneqq \left(1 + X + \Sets_{\geq 2}\left(Z\right)\right) %
  \cdot(\Sets\circ\Tuft) - 1.
\]
\end{example}

The following proposition summarizes how the operations on species
interact with passing to the different types of generating series.
\begin{proposition}\label{prop:series}
  Let $F$ and $G$ be species.  Then,
  \begin{align*}
    (F+G)(x) &= F(x) + G(x), & 
    (F\cdot G)(x) &= F(x) G(x), \\
    \widetilde{(F+G)}(x) &= \widetilde F(x) + \widetilde G(x), & 
    \widetilde{(F\cdot G)}(x) &= \widetilde F(x)\widetilde G(x), \\
    Z_{F+G} &= Z_F + Z_G & 
    Z_{F\cdot G} &= Z_F Z_G.
  \end{align*}
  Moreover, provided that the restriction of $G$ to the empty set is
  zero, $(F\circ G)(x) = F(G(x))$ and $\widetilde{(F\circ G)}(x)$ is
  obtained from $Z_F$ by substituting $\widetilde G(x^k)$ for
  $p_k$.\footnote{More generally, $Z_{F\circ G}$ is known as the
    plethysm of the symmetric functions $Z_F$ and $Z_G$, but we will
    not use this notion.}
\end{proposition}

A species $F$ is \Dfn{molecular} if any two $F$-structures are
isomorphic.  The zero species $0$, producing no structure for any set
of labels, is trivially molecular.
\begin{example}
Given a graph~$R$, define a species $\mathcal R$ such that the structures $\mathcal R[U]$ are the graphs isomorphic to~$R$ with vertex set~$U$.  Then~$\mathcal R$ is molecular.  Note that the two species obtained in this way from the edgeless graph with~$n$ vertices and the complete graph with~$n$ vertices are equal.
\end{example}

Any species can be written as a sum of molecular species, and this
\Dfn{molecular decomposition} is uniquely determined.  Put differently, two species are equal if and only if the coefficients in their molecular expansions coincide.  We define a
\Dfn{virtual ($k$-sort) species} as a formal difference of two
($k$-sort) species, up to the natural equivalence relation.  Thus, the coefficients of the molecular expansion of a virtual species can be arbitrary integers.

Similar to the common practice of providing the first few Taylor coefficients of a formal power series, it may be helpful for developing some intuition to display the first few terms of the molecular decomposition of a species.

Virtual species allow us to define the \Dfn{multiplicative inverse} and \Dfn{compositional inverse} of a species, subject to a natural condition.
\begin{theorem}[\protect{\cite[Exc.~2.5.7., Prop.~2.5.19]{MR1629341}}]
  A virtual species has a multiplicative inverse if its
  restriction to the empty set equals the species~$1$ or~$-1$.  A $1$-sort virtual
  species $F$ has a compositional inverse $F^{(-1)}$ if its restriction
  to the empty set is the zero species and its restriction to sets of cardinality~$1$ is the species~$X$ or~$-X$.
\end{theorem}
\begin{example}\label{ex:inverses}~
\begin{enumerate}
    \item The species of sets $\Sets$ satisfies the condition of the first part of the theorem, so its multiplicative
inverse exists.  By~\Cref{ex: productofsets}, it is given by $\Sets(-X)$.  Its molecular decomposition begins
with
\[
1 - X - \Sets_2+X^2 - \Sets_3+2X\Sets_2 - X^3 - \Sets_4+2X\Sets_3+\Sets_2^2-3 X^2\Sets_2+X^4+\dots
\]
\item The molecular decomposition of the compositional inverse of the
species of non-empty sets $\Sets_{\geq 1}$, also known as the
\Dfn{combinatorial logarithm $\Omega$}, begins with
\[
X - \Sets_2 - \Sets_3+X \Sets_2 - \Sets_4+ \Sets_2\circ\Sets_2 + X \Sets_3-X^2 \Sets_2 +\dots
\]
The isomorphism type generating series of $\Omega$ is
$\widetilde{\Omega}(x)=x-x^2$, by the following reasoning.  Note that
$\big(\Sets_{\geq1}(X-X^2)\big)^\sim(x)=\widetilde\Sets(x)/\widetilde\Sets(x^2)-1=\frac{1-x^2}{1-x}-1=x$.
Taking into account that the isomorphism type generating series of
a composition $F\circ G$ only depends on $\widetilde G(x)$ by
\Cref{prop:series}, we obtain
\[\widetilde\Omega(x)=\big(\Omega\circ\Sets_{\geq
  1}(X-X^2)\big)^\sim(x)=x-x^2.\]

\item The compositional inverse of $X-X^2$ is the species of ordered rooted trees $\mathcal A$
with labels on all vertices~\cite[Eq.~(1.1.27)]{MR1629341}. This follows from the decomposition of an ordered rooted tree as a root and a tuple of ordered rooted trees, which gives the equality of species $\mathcal A = X/(1-\mathcal A).$ Rearranging gives $\mathcal A - \mathcal A^2 = X$.
\end{enumerate}
\end{example}

\section{A Decomposition of Graphs}\label{sec:patch decomp}

In this section, we generalize a functional equation discovered by Gessel and Li to take siblings into account.  Using the
notation introduced in \Cref{sec:species}, the two intermediate
steps~\cite[Eq.~(2.2)]{read1989enumeration} and \cite[Thm.~2.2]{MR2739506}, and \cite[Thm.~3.1]{MR2739506} on which Gessel and Li's argument is based, read as follows.
\begin{proposition}[cf.~Read 1989]\label{prop:graphdecomp}
  The species of connected point distinguishing graphs $\comating^\connected$ and the species of connected graphs
  $\Graphs^\connected$ are related by the equality
  \[
    \comating^\connected\circ\Sets_{\geq 1} = \Graphs^\connected.
  \]
\end{proposition}

\begin{proposition}[Gessel \& Li~2011]\label{prop:petalsY}
    The species of connected graphs other than $K_2$, with leaves of sort $X$ or $Y,$ and other vertices of sort $X$ equals
    \[
    \decorated^\connected(X, X+Y, X) = \Graphs^\connected\big(X\Sets(Y)\big) - XY - \Sets_2(X).
    \]
\end{proposition}
We remark that the two error terms on the right-hand side differ from
the result of Gessel and Li only because we restrict the species
$\decorated^\connected$ to structures whose size is not equal to $2$.

Combining \Cref{prop:graphdecomp,prop:petalsY} we obtain
the equation
\begin{equation*}
    \decorated^\connected(X, X+Y, X) = \comating^\connected\circ\Sets_{\geq 1}\big(X\Sets(Y)\big) - XY - \Sets_2(X),
\end{equation*}
which is the key ingredient of Gessel and Li's proof of
\Cref{thm:KilGL}.  Our goal is to generalize this decomposition so
that we can also keep track of siblings.  Recall that $\Patch$ denotes
the species of patches defined in \Cref{exam: cultivations}.
\begin{theorem}[Patch Decomposition of Graphs]\label{thm:patch-decomposition}
  The species of connected graphs other than $K_2$, with leaves of sort $X$ or $Y,$ siblings of sort $Z$, and other vertices of sort $X$ equals
  \[
    \decorated^\connected(X, X+Y, Z) = \comating^\connected\circ\Patch - XY - \Sets_2(Z).
  \]
\end{theorem}

This functional equation reflects the fact that the vertices of any graph can be partitioned into leaves, siblings and other vertices.  In principle, this partition can be refined to also distinguish non-leaf vertices which share the same open neighborhood with some other vertex~\cite{BurchardtHahn}. It would be interesting to find a functional equation that exploits this refinement. 

At the heart of the patch decomposition are the underlying point distinguishing graph and the corresponding patch partition of a $\decorated^\connected(X, X+Y, Z)$-structure $G$. Before stating precise definitions, we provide an informal description. The underlying point distinguishing graph of $G$ is obtained by removing all leaves of sort $Y,$ and then contracting each group of siblings in the
resulting graph, ignoring whether vertices are of sort $X$ or $Z$, to
a single vertex.  Each block of the patch partition corresponds to a vertex in this graph and consists of the siblings contracted to this vertex together with the adjacent leaves that were removed in the first step.
A visualisation is provided in \Cref{fig:FloralDecomposition}.
\begin{figure}[h]
    \centering
    \begin{tikzpicture}
        \node (X1) at (0,0) {$X$};
        \node (X2) at (1,0) {$X$};
        \node (X3) at (2,0) {$X$};
        \node (X4) at (2,1) {$X$};
        \node (X5) at (3,0) {$X$};
        \node (Y21) at (1,1) {$Y$};
        \node (Y41) at (1.5,2) {$Y$};
        \node (Y42) at (2.5,2) {$Y$};
        \node (Z1) at (.5, -1) {$Z$};
        \node (Z2) at (1.5, -1) {$Z$};
        \node (X6) at (.5, -2) {$X$};
        \node (X7) at (1.5, -2) {$X$};
        \node (Y61) at (.5, -3) {$Y$};
        \node (Y71) at (1.5, -3) {$Y$};
        \node (Y72) at (2.5, -2) {$Y$};

        \draw (X1) -- (X2);
        \draw (X2) -- (X3);
        \draw (X2) -- (X4);
        \draw (X4) -- (X3);
        \draw (X4) -- (X5);
        \draw (X5) -- (X3);
        \draw (X2) -- (Y21);
        \draw (X4) -- (Y41);
        \draw (X4) -- (Y42);
        \draw (X6) -- (Y61);
        \draw (X7) -- (Y71);
        \draw (X7) -- (Y72);
        \draw (X2) -- (Z1);
        \draw (X2) -- (Z2);
        \draw (X2) -- (X6);
        \draw (X2) -- (X7);
        \draw (Z1) -- (Z2);
        \draw (Z1) -- (X6);
        \draw (Z1) -- (X7);
        \draw (Z2) -- (X6);
        \draw (Z2) -- (X7);
        \draw (X7) -- (X6);

        \draw[style=dashed] (0.2, 0.2) to[out=-45, in=45] (0.2, -0.2) to[out=-135, in=-45] (-0.2, -0.2) to [out=135, in=-135] (-0.2, 0.2) to[out=45, in=135] (0.2, 0.2);
        \draw[style=dashed] (0.8, -0.2) to[out=-45, in = -135] (1.2, -0.2) to[out=45, in=-45] (1.2, 1.2) to[out=135, in=45] (0.8, 1.2) to [out=-135, in=135] (0.8, -0.2);
        \draw[style=dashed] (1.3, 2.2) to [out=45, in=135] (2.7, 2.2) to[out=-45, in=90] (2.3, 0.8) to[out=-90, in=45] (2.2, -.2) to[out = -135, in=-45] (1.8, -.2) to[out=135, in=-90] (1.7, 0.8) to [out=90, in=-135] (1.3, 2.2);
        \draw[style=dashed] (3.2, 0.2) to[out=-45, in=45] (3.2, -0.2) to[out=-135, in=-45] (2.8, -0.2) to [out=135, in=-135] (2.8, 0.2) to[out=45, in=135] (3.2, 0.2);
        \draw[style=dashed] (.3, -.8) to[out=45, in=135] (1.8, -.8) to [out=-45, in=135] (2.7, -1.8) to [out=-45, in=45] (2.7, -2.2) to [out=-135, in=45] (1.7, -3.2) to[out=-135, in=-45] (0.3, -3.2) to [out=135, in=-135] (.3, -.8);
    \end{tikzpicture}
    \hspace{2cm}
    \begin{tikzpicture}
        \node[vertex] (V1) at (0,0) {};
        \node[vertex] (V2) at (1,0) {};
        \node[vertex] (V3) at (2,0) {};
        \node[vertex] (V4) at (3,0) {};
        \node[vertex] (V5) at (1,-1) {};

        \draw (V1) -- (V2) -- (V3) -- (V4);
        \draw (V2) -- (V5);

        \node (X1) at (0,1) {$X$};
        \node (X2) at (1,1) {$X$};
        \node (X3) at (2,1) {$X$};
        \node (X4) at (2,2) {$X$};
        \node (X5) at (3,1) {$X$};
        \node (Y21) at (1,2) {$Y$};
        \node (Y41) at (1.5,3) {$Y$};
        \node (Y42) at (2.5,3) {$Y$};

        \draw (X2) -- (Y21);
        \draw (X4) -- (Y41);
        \draw (X4) -- (Y42);

        \node (Z1) at (2, -1) {$Z$};
        \node (Z2) at (3, -1) {$Z$};
        \node (X6) at (2, -2) {$X$};
        \node (X7) at (3, -2) {$X$};
        \node (Y61) at (2, -3) {$Y$};
        \node (Y71) at (3, -3) {$Y$};
        \node (Y72) at (4, -2) {$Y$};

        \draw (X6) -- (Y61);
        \draw (X7) -- (Y71);
        \draw (X7) -- (Y72);

        \draw[style=dashed] (0.2, 1.2) to[out=-45, in=45] (0.2, 0.8) to[out=-135, in=-45] (-0.2, 0.8) to [out=135, in=-135] (-0.2, 1.2) to[out=45, in=135] (0.2, 1.2);
        \draw[style=dashed] (0.8, 0.8) to[out=-45, in = -135] (1.2, 0.8) to[out=45, in=-45] (1.2, 2.2) to[out=135, in=45] (0.8, 2.2) to [out=-135, in=135] (0.8, 0.8);
        \draw[style=dashed] (1.3, 3.2) to [out=45, in=135] (2.7, 3.2) to[out=-45, in=90] (2.3, 1.8) to[out=-90, in=45] (2.2, .8) to[out = -135, in=-45] (1.8, .8) to[out=135, in=-90] (1.7, 1.8) to [out=90, in=-135] (1.3, 3.2);
        \draw[style=dashed] (3.2, 1.2) to[out=-45, in=45] (3.2, .8) to[out=-135, in=-45] (2.8, .8) to [out=135, in=-135] (2.8, 1.2) to[out=45, in=135] (3.2, 1.2);
        \draw[style=dashed] (1.8, -.8) to[out=45, in=135] (3.3, -.8) to [out=-45, in=135] (4.2, -1.8) to [out=-45, in=45] (4.2, -2.2) to [out=-135, in=45] (3.2, -3.2) to[out=-135, in=-45] (1.8, -3.2) to [out=135, in=-135] (1.8, -.8);

        \node (A1) at (0, .5) {$\downarrow$};
        \node (A2) at (1, .5) {$\downarrow$};
        \node (A3) at (2, .5) {$\downarrow$};
        \node (A4) at (3, .5) {$\downarrow$};
        \node (A5) at (1.5, -1) {$\leftarrow$};
    \end{tikzpicture}
    \caption{Decomposing a $\decorated^\connected(X, X+Y, Z)$-structure into its underlying point distinguishing graph and an assignment of patches to its vertices.}
    \label{fig:FloralDecomposition}
\end{figure}

Note that the underlying point distinguishing graph of a
$\decorated^\connected(X, X+Y, Z)$-structure may have leaves, but
cannot have siblings.  Consider, for example, the path graph with
$n\geq 4$ vertices with both leaves of sort $Y$ and inner vertices of
sort $X$.  If $n\geq 5$, the underlying point distinguishing graph is the path
graph with $n-2$ vertices.  By contrast, if $n=4$, removing the leaves we
obtain $K_2$, which consists of two siblings.  Thus, the underlying
point distinguishing graph is a single vertex.

To define the underlying point distinguishing graph and the patch partition precisely, we introduce two maps.  For a set $L$ of leaves of a graph $G$, let $\lambda_L(G)$ be the graph obtained from $G$ by removing $L$.  Moreover, let $\rho(G)$ be the graph obtained from $G$ by contracting each group of siblings to a single vertex.  Note that $\rho(G)$ is point distinguishing.

Given $G\in\decorated^\connected(X, X+Y, Z)[U]$, with $U=(U_1,U_2,U_3)$, its \Dfn{patch partition} $\left(\pi, \left(P_B\right)_{B\in\pi}\right)$ consists of a set partition $\pi\in\Par[U]$ and, for each block $B\in\pi$, a patch $P_B\in\Patch[B]$.  Each block $B$ corresponds to a triple $(A,L,R)$ with $A\subseteq U_1$, $L\subseteq U_2$, $R\subseteq U_3$, such that $A + R$ is a maximal group of siblings of $\lambda_{U_2}(G)$, and $L$ are the leaves of sort $Y$ adjacent in $G$ to a vertex in $A$.  The patch $P_B$ has leaves $L$ and roots $A+R$. The adjacencies in the patch are precisely the adjacencies in $G$ between vertices in $A$ and $L$.
Finally, the \Dfn{underlying point distinguishing graph} of $G$ is $\rho\circ\lambda_{U_2}(G)\in\comating^\connected[\pi]$.

\begin{proof}[Proof of \Cref{thm:patch-decomposition}, the Patch Decomposition]
  Let $U=(U_1, U_2, U_3)$ be a $3$-set of labels.  By the definition
  of composition, a structure in $(\comating^\connected\circ\Patch)[U]$ is a triple
  $\left(\pi, D, \left(P_B\right)_{B\in\pi}\right)$ with
  $\pi\in\Par[U]$, $D\in\comating^\connected[\pi]$ and $P_B\in\Patch[B]$ for each
  $B\in\pi$.

  We have to exhibit an equivariant bijection between these triples
  and graphs $G\in\decorated^\connected(X, X+Y, Z)[U]$, whose leaves can be of
  sort $X$ or $Y,$ with exceptions for
  $\lvert U_1\rvert = \lvert U_2\rvert = 1$, $U_3=\emptyset$ and for
  $U_1=U_2=\emptyset$, $\lvert U_3\rvert=2$.

  Given the triple, the graph $G$ is obtained from the disjoint union
  of the patches $(P_B)_{B\in\pi}$ as follows.  For each block $B\in\pi$, we
  add edges between all the roots of $P_B$. Additionally, we add
  edges between all roots of $P_B$ and $P_{B'}$ if $B$ and $B'$ are
  adjacent in $D$.

  Let us first check that this yields an element $G$ in
  $\decorated^\connected(X, X+Y, Z)[U]$.
  By the definition of a patch, only leaves may be of sort $Y$.

  Since a patch $P_B$ contains either no,
  or at least two vertices of sort $Z$, vertices of sort $Z$ cannot
  be leaves in $G$, provided that there is at least one other root in
  the same patch or in a patch adjacent to $P_B$.  The exceptional
  case of an $\comating^\connected\circ\Patch$-structure whose image has leaves of sort $Z$ and thus does not
  correspond to a $\decorated^\connected(X, X+Y, Z)$-structure is illustrated on
  the left of \Cref{fig:exception}.

  Thus, it remains to show that precisely the vertices of sort $Z$
  within a block are siblings in $G$.  Two vertices in the same block $B$
  are siblings in $G$ if and only if they are both of sort $Z$,
  unless $D$ is a single vertex and $P_B$ is a single tuft with a
  single leaf. The exceptional case of an $\comating^\connected\circ\Patch$-structure whose image has siblings of sort other than $Z$ is illustrated on the right of
  \Cref{fig:exception}.
  \begin{figure}[h]
    \centering
    \begin{tikzpicture}
        \node (Z1) at (-.5,1) {Z};
        \node (Z2) at (.5,1) {Z};
        \node[vertex] (V1) at (0, -.2) {};
        \draw[style=dashed] (-.7, 1.2) to[out=45, in=135] (.7, 1.2) to[out=-45, in=45] (.7, 0.8) to [out=-135, in=-45] (-.7, .8) to[out=135, in=-135] (-.7, 1.2);
        \node (A1) at (0, .3) {$\downarrow$};
        \node (A2) at (1.5, .6) {$\mapsto$};
        \node (Z3) at (2, .6) {$Z$};
        \node (Z4) at (3, .6) {$Z$};
        \draw (Z3) -- (Z4);
    \end{tikzpicture}
    \hspace{1cm}
    \begin{tikzpicture}
        \node (X1) at (-.5,1) {X};
        \node (Y1) at (.5,1) {Y};
        \draw (X1) -- (Y1);
        \node[vertex] (V1) at (0, -.2) {};
        \draw[style=dashed] (-.7, 1.2) to[out=45, in=135] (.7, 1.2) to[out=-45, in=45] (.7, 0.8) to [out=-135, in=-45] (-.7, .8) to[out=135, in=-135] (-.7, 1.2);
        \node (A1) at (0, .3) {$\downarrow$};
        \node (A2) at (1.5, .6) {$\mapsto$};
        \node (X2) at (2, .6) {$X$};
        \node (Y2) at (3, .6) {$Y$};
        \draw (X2) -- (Y2);
    \end{tikzpicture}
    \caption{The two types of $\comating^\connected\circ\Patch$-structures that do not correspond to $\decorated^\connected(X, X+Y, Z)$-structures.}
    \label{fig:exception}
  \end{figure}

  So it suffices to show that $u\in B$ and $u'\in B'$ with $B\neq B'$
  are not siblings in $G$.  If one of them is a leaf of its patch,
  its only neighbor lies in the same block, so they are not siblings.
  We may therefore assume that both are roots.  If $B$ and $B'$ are not adjacent in $D$, then
  $u$ and $u'$ are not adjacent in $G$ and therefore not siblings.
  If $B$ and $B'$ are adjacent in $D$, the closed neighborhoods of
  $B$ and $B'$ must be distinct because $D$ has no siblings.  Without
  loss of generality, suppose that $B''$ is a neighbor of $B$ but
  not of $B'$.  Then, any root in $P_{B''}$ is a neighbor of $u$ but
  not of $u'$, so $u$ and $u'$ are not siblings.

  Next, we check that the patch partition of $G$ is $\left(\pi, \left(P_B\right)_{B\in\pi}\right)$, and the underlying point distinguishing graph of $G$ equals $D$.  Recall that $\lambda_{U_2}(G)$ is obtained from $G$ by removing the leaves of sort $Y$.  Since the adjacencies in $\lambda_{U_2}(G)$ of roots of sort $X$ and roots of sort $Z$ within a block are exactly the same, each maximal group of siblings in $\lambda_{U_2}(G)$ consists exactly of the roots of one block.  It follows that $D$ equals $\rho\circ\lambda_{U_2}(G)$.

  Conversely, the patch partition and the underlying point distinguishing graph determine, for each $\decorated$-structure, a triple $\left(\pi, D, \left(P_B\right)_{B\in\pi}\right)$. It is clear that assembling the patches in the patch partition according to the adjacencies of the underlying point distinguishing graph reconstructs the original graph.  Hence, we have a bijection between $\comating^\connected\circ\Patch$-structures and $\decorated^\connected$-structures. The bijection intertwines with relabeling, since its definition does not involve any choice depending on the labels.
\end{proof}

\section{Bi-Point-Determining Graphs}\label{sec:bipoint}
Gessel and Li also study bi-point determining graphs.  In particular, they provide the following expression for the species of connected bi-point determining graphs~\cite[Cor.~4.7]{MR2739506}.
\begin{theorem}[Gessel \& Li~2011]\label{thm:bipointdetermining}
  The species $\mathcal B^\connected$ of connected bi-point determining graphs
  equals
  \[
    \mathcal B^\connected = %
    \Graphs^\connected(2\Omega - X) %
    - \Omega + X.
  \]
\end{theorem}
Combining the patch decomposition with this theorem we can deduce \Cref{cor:bipointdetermining}, that is, the number of
unlabeled connected leafless point distinguishing graphs coincides with
the number of bi-point-determining graphs.  It would be interesting to find an explicit bijection demonstrating this fact.

\begin{proof}[Proof of \Cref{cor:bipointdetermining}]
\sloppypar
  The species of connected leafless point distinguishing graphs is equal to
  $\decorated^\connected(X, 0, 0)$.  We want to compare its isomorphism type generating series with the isomorphism type generating series of $\mathcal B^\connected = %
    \Graphs^\connected(2\Omega - X) %
    - \Omega + X.$ Setting $Y=-X$ and $Z=0$ in the Patch Decomposition from 
  \Cref{thm:patch-decomposition} and using \Cref{prop:graphdecomp} we
  obtain
  \begin{align*}
    \decorated^\connected(X, 0, 0) %
    &= \comating^\connected\circ\Patch(X, -X, 0) + X^2\\
    &= \Graphs^\connected\circ\Omega\circ\Patch(X, -X, 0) + X^2.
  \end{align*}

By \Cref{prop:series}, the isomorphism type generating series of
  two compositions $F\circ G$ and $F\circ H$ are equal if
  $\widetilde G(x)=\widetilde H(x)$.  Thus, comparing with \Cref{thm:bipointdetermining}, it suffices to show that
  \begin{align*}
    \big(2\Omega - X\big)^\sim(x) &= \big(\Omega\circ\Patch(X, -X, 0)\big)^\sim(x)\quad\text{and}\\
    \big(-\Omega + X\big)^\sim(x) &= x^2.
  \end{align*}
  
Recall from \Cref{ex:inverses}~(2) that $\widetilde\Omega(x)=x-x^2$, which implies the second equality. To show the first equality we will additionally use
  $\big(\Sets(X-X^2)\big)^\sim(x)=1+x$ and
  $\big(X\Sets_{\geq1}(-X)\big)^\sim(x)=-x^2$ from the same example.
  
  By definition,
  $\Patch(X, -X, 0) %
  = (1 + X)\cdot\Sets\big(X\Sets_{\geq1}(-X)\big)-1$.  Therefore, again substituting different species with the same isomorphism type generating series in the equation,
  \begin{align*}
    \big(\Patch(X, -X, 0)\big)^\sim(x) %
    &= \big(\Sets(X-X^2)\cdot\Sets(-X^2)-1\big)^\sim(x)\\
    &= \big(\Sets_{\geq1}(X-2X^2)\big)^\sim(x).
  \end{align*}
  This implies
  $\big(\Omega\circ\Patch(X, -X,
  0)\big)^\sim(x)=x-2x^2=(2\Omega-X)^\sim(x)$.
\end{proof}

\section{Symmetry of the Sibling Number and the Tuft Number}
\label{sec:symmetry}

In this section, we prove \Cref{thm:main}.  Let us first refine the
Patch Decomposition from \Cref{thm:patch-decomposition} to take the
sibling number into account.  To do so, let $\decorated^{\connected, \sibling}$ be
the subspecies of the species of \decoratedgraphs $\decorated^\connected$ with
sibling number at most~$\sibling$ and let $\Patch^{\sibling}$ be the
subspecies of the species of patches with sibling number at
most~$\sibling$:
\[
  \Patch^{\sibling}(X, Y, Z) %
  \coloneqq \left(1 + X + \Sets_{2\leq\bullet\leq 1+\sibling}(Z)\right) %
  \cdot(\Sets\circ\Tuft) - 1.
\]
Moreover, we use the Iverson bracket
$[s>0]\coloneqq \begin{cases} 1 & \text{if $s>0$},\\ 0 & \text{otherwise.}\end{cases}$
\begin{corollary}\label{thm:patch-decomposition-siblings}   The species of connected graphs other than $K_2$, with leaves of sort $X$ or $Y,$ siblings of sort $Z$, other vertices of sort~$X$, and restricted sibling number equals
  \[
    \decorated^{\connected, \sibling}(X, X+Y, Z) %
    = \comating^\connected\circ\Patch^{\sibling} %
    - XY - [\sibling > 0] \Sets_2(Z).
  \]
\end{corollary}
\begin{proof}
  Let $U=(U_1, U_2, U_3)$ be a $3$-set of labels.  By the Patch Decomposition from 
  \Cref{thm:patch-decomposition} we have an equivariant bijection
  between graphs $G\in\decorated^\connected(X, X+Y, Z)[U]$ and triples
  $\left(\pi, D, \left(P_B\right)_{B\in\pi}\right)$ with
  $\pi\in\Par[U]$, $D\in\comating^\connected[\pi]$ and $P_B\in\Patch[B]$ for
  each $B\in\pi$, with the exception of
  $(\comating^\connected\circ\Patch)$-structures of size~$2$.

  We have shown that each group of siblings of $G$ correspond to the
  roots of sort $Z$ in a patch.  Thus, the sibling number of $G$
  is at most $\sibling$ if and only if $P_B\in\Patch^{\sibling}[B]$ for
  each $B\in\pi$.

  The exceptional
  $(\comating^\connected\circ\Patch^{\sibling})$-structures of size
  $2$ are depicted in \Cref{fig:exception}.  If $\sibling=0$, the
  unique $(\comating^\connected\circ\Patch^{\sibling})$-structure of
  size $2$ is a single tuft.  Otherwise, if $\sibling>0$, there is an
  additional structure consisting of two vertices of sort $Z$.
\end{proof}

It remains to take into account the tuft number.  One might expect
that this can be achieved by restricting the tuft number of each
tuft appearing in $\Patch^{\sibling}$.  This is not the case, because
$\decorated^{\connected, \sibling}(X, X+Y, Z)$ may have leaves of sort $X$.

Instead, we first reduce the problem to the enumeration of
\decoratedgraphs with bounded sibling number and tuft number
equal to~$1$, using a suitable substitution of virtual species.  In a
second step, we replace each tuft in a graph with tuft
number~$1$ with a tuft having at most $\tuft$ leaves.

\begin{lemma}\label{lem:decoratedgeneric}
  Let $H$ be a set of connected graphs other than $K_2$, closed under
  replacing each set of leaves adjacent to the same vertex with any
  positive number of leaves adjacent to this vertex, provided that the resulting graph is different from $K_2$.  Let
  $\decoratedgeneric$ be the subspecies of $\decorated^\connected$ whose
  structures are isomorphic to a graph in $H$.  Let
  $\decoratedgeneric^\tuft$ be the subspecies of
  $\decoratedgeneric$ whose structures have tuft number at most
  $\tuft$.
  Finally, let
  $\EInProofOfThm_\tuft = \Sets_{1\leq\bullet \leq \tuft}(Y)$.

  Then, if $H$ contains the star graphs with at least three vertices,
  \begin{align*}
    \decoratedgeneric^\tuft(X, Y, Z) %
    &= \decoratedgeneric(X, \Omega(\EInProofOfThm_\tuft), Z) + X\Omega(\EInProofOfThm_\tuft) - [\tuft > 0] X Y,\\
    \intertext{whereas otherwise,}
    \decoratedgeneric^\tuft(X, Y, Z) %
    &= \decoratedgeneric(X, \Omega(\EInProofOfThm_\tuft), Z).
  \end{align*}
\end{lemma}
\begin{proof}
  Suppose first that $H$ contains the star graphs with at least three vertices.  Then any graph in $\decoratedgeneric^\tuft[U]$, with the
  exception of the star graphs, can be obtained from a graph in
  $\decoratedgeneric^1$ by replacing each tuft
  with a tuft having at least one and at most $\tuft$ leaves.  The star graphs
  in $\decoratedgeneric^\tuft[U]$ consist of a single
  vertex of sort $X$ together with at least two and at most
  $\tuft$ leaves of sort $Y$.  Therefore, setting
  $\EInProofOfThm_\tuft \coloneqq\Sets_{1\leq\bullet \leq \tuft}(Y)$ and using
  that
  $\Sets_{2\leq\bullet \leq \tuft}(Y) = \EInProofOfThm_\tuft - [\tuft
  > 0] Y,$ we obtain
  \begin{equation}\label{eq:reduce}
    \decoratedgeneric^\tuft(X, Y, Z) %
    = \decoratedgeneric^1(X, \EInProofOfThm_\tuft, Z) %
    + X \EInProofOfThm_\tuft - [\tuft > 0] X Y.
  \end{equation}
  Setting $\tuft=\infty$ in this equation yields
  \[
    \decoratedgeneric(X, Y, Z) %
    = \decoratedgeneric^1(X, \Sets_{\geq 1}(Y), Z) %
    + X \Sets_{\geq 1}(Y) - X Y.
  \]
  Substituting $\Omega(\EInProofOfThm_\tuft)$ for $Y$ we obtain
  \begin{equation}\label{eq:invert}
    \decoratedgeneric(X, \Omega(\EInProofOfThm_\tuft), Z) %
    = \decoratedgeneric^1(X, \EInProofOfThm_\tuft, Z) %
    + X \EInProofOfThm_\tuft - X \Omega(\EInProofOfThm_\tuft).
  \end{equation}
  Therefore,
  \begin{align*}
    \decoratedgeneric^\tuft(X, Y, Z)
    &\stackrel{\eqref{eq:reduce}}{=} \decoratedgeneric^1(X, \EInProofOfThm_\tuft, Z) %
      + X \EInProofOfThm_\tuft - [\tuft > 0] X Y\\
    &\stackrel{\eqref{eq:invert}}{=} \decoratedgeneric(X, \Omega(\EInProofOfThm_\tuft), Z) %
      + X\Omega(\EInProofOfThm_\tuft) - [\tuft > 0] X Y.
  \end{align*}
  If $H$ does not contain the star graphs on at least three vertices the computation simplifies, and we
  obtain
  \[
    \decoratedgeneric^\tuft(X, Y, Z) =
    \decoratedgeneric(X, \Omega(\EInProofOfThm_\tuft), Z). \qedhere
  \]
\end{proof}

\begin{lemma}\label{lem:Pst}
  Let $\mathcal Q^{\sibling,\tuft}$ be the virtual species
  \[
    \mathcal Q^{\sibling,\tuft}(X,Y,Z)\coloneqq %
    \big(1 + X + \Sets_{2\leq\bullet \leq 1+\sibling}(Z)\big) %
    \Sets%
    \Big(X\big(\Sets_{\leq \tuft}(Y)\Sets(-X)-1\big)\Big) %
    - 1.
  \]
  Then
  \[
    \Patch^{\sibling}\big(X, \Omega(\EInProofOfThm_\tuft) - X, Z\big) %
    = \mathcal Q^{\sibling,\tuft}(X,Y,Z).
  \]
\end{lemma}
\begin{proof}
  Recall that $\Patch^{\sibling}(X, Y, Z) %
  = \left(1 + X + \Sets_{2\leq\bullet\leq 1+\sibling}(Z)\right) %
  \cdot(\Sets\circ\Tuft) - 1$ with
  $\Tuft(X, Y) = X\Sets_{\geq 1}(Y)$.  Thus, we compute using \Cref{ex: productofsets}
  \begin{align*}
    \Sets_{\geq 1}\big(\Omega(\EInProofOfThm_\tuft) - X\big)
    & =\Sets\big(\Omega(\EInProofOfThm_\tuft) - X\big)-1\\
    & =\Sets\big(\Omega(\EInProofOfThm_\tuft)\big)\Sets(-X)-1\\
    & =(1 + \EInProofOfThm_\tuft)\Sets(-X)-1\\
    & =\Sets_{\leq t}(Y)\Sets(-X)-1,
  \end{align*}
  which implies the claim.
\end{proof}

\begin{theorem}\label{thm:enumeration}The species of connected graphs other than $K_2$, with leaves of sort $Y,$ siblings of sort $Z$, other vertices of sort~$X$, and with restricted sibling and tuft number equals
  \[
    \decorated^{\connected, \sibling,\tuft}(X, Y, Z) %
    = \comating^\connected\circ\mathcal Q^{\sibling,\tuft}(X, Y, Z) %
    + X^2 - [\tuft > 0] X Y - [\sibling > 0] \Sets_2(Z).\\
  \]
\end{theorem}
\begin{proof}
  We apply \Cref{lem:decoratedgeneric} in the case where $\decoratedgeneric = \decorated^{\connected, \sibling}$ is the species of
  connected graphs other than $K_2$ with sibling number at most $\sibling$.  The closure property holds because replacing a positive number of leaves adjacent to a vertex by another positive number does not affect the sibling number, unless the resulting graph is a $K_2$.
  
  The star graphs with at
  least three vertices have sibling number zero, so the first equation in \Cref{lem:decoratedgeneric} applies.
  After setting $Y = \Omega(\EInProofOfThm_\tuft) - X$ in
  \Cref{thm:patch-decomposition-siblings}, the claim follows from \Cref{lem:Pst}.
\end{proof}

We can now prove that the sibling and the tuft number have joint
symmetric distribution on the set of unlabeled connected graphs
with~$n$ vertices.  In fact, we show that the isomorphism type generating series
of $\decorated^{\connected, \sibling,\tuft}(X, X, X)$ is symmetric in
$\sibling$ and $\tuft$.
\begin{lemma}\label{prop:argsym}
  The isomorphism type generating series of the virtual species
  $\mathcal Q^{\sibling,\tuft}(X,X,X)$ is
  \[
    \widetilde{\mathcal Q^{\sibling,\tuft}}(x) =
    \frac{(1-x^{\sibling+2})(1-x^{\tuft +2})}{1-x} - 1.
  \]
  The isomorphism type generating series of the virtual species
  \[
    X^2 - [\tuft > 0] X^2 - [\sibling > 0] \Sets_2
  \]
  is $x^2 - [\tuft > 0] x^2 - [\sibling > 0] x^2$.

  In particular, these generating series are symmetric in $\sibling$
  and $\tuft$.
\end{lemma}
\begin{proof}
  The first factor in $\mathcal Q^{s,t}(X,X,X)+1$ is
  $1+X+\Sets_{2\leq\bullet \leq 1+\sibling} = \Sets_{\leq 1+\sibling}$
  and its isomorphism type generating series is
  \[
    \widetilde{\Sets_{\leq 1+\sibling}}(x) = (1-x^{\sibling+2})/(1-x).
  \]
  To compute the isomorphism type generating series of the second
  factor, we first set
  $\mathcal F^\tuft \coloneqq X(\Sets_{\leq \tuft}(X)\cdot\Sets(-X)-1)$
  and find that
  \[
    \widetilde{\mathcal F^\tuft}(x) %
    = x\left(\frac{1-x^{\tuft+1}}{1-x}\cdot (1-x)-1\right) %
    =-x^{\tuft+2}.
  \]
  We now recall that the cycle index series of $\Sets$ equals
  $Z_{\Sets} = \exp\left(\sum_{k>0}\frac{p_k}{k}\right)$.
  Substituting $p_k=-x^{k(t+2)}$ we obtain, in accordance with
  \Cref{prop:series},
    \[
      \widetilde{\Sets(\mathcal F^\tuft)}(x) %
      = \exp\left(-\sum_{k>0}\frac{x^{k(t+2)}}{k}\right) = 1-x^{t+2},
    \]
    which finishes the computation of
    $\widetilde{\mathcal Q^{\sibling,\tuft}}(x)$.
\end{proof}

\begin{proof}[Proof of \Cref{thm:main}]
  By \Cref{thm:enumeration} and \Cref{prop:argsym}, the isomorphism
  type generating series of $\decorated^{\connected, \sibling,\tuft}(X,X,X)$
  is symmetric in $\sibling$ and $\tuft$.

  Since the species of graphs other than $K_2$ with sibling number $\sibling$ and
  tuft number $\tuft$ equals
  \[
    \decorated^{\connected, \sibling,\tuft} - \decorated^{\connected, \sibling-1,\tuft}
    - \decorated^{\connected, \sibling,\tuft - 1} +
    \decorated^{\connected, \sibling-1,\tuft-1},
  \]
  the claim follows for $n\neq 2$.  For $n=2$, we have
  $\sibling(K_2)=\tuft(K_2)=1$.
\end{proof}
\section{The Reduction of a Graph}
\label{sec:reduction}
In this section, we show that the reduction of a graph, as defined in
the introduction, and the reduction of the same graph with some
leaves removed coincide.  In particular, we show that the reduction of the underlying
point distinguishing graph equals the reduction of the original graph.

In fact, we will prove a slight generalization of this statement.
Recall from \Cref{sec:patch decomp} that $\lambda_L(G)$ is the graph obtained from $G$ by removing a set of leaves $L$, and $\rho(G)$ is the graph obtained from
$G$ by contracting each group of siblings to a single vertex.  For
convenience, let $\lambda$ without subscript denote the map removing all
leaves of a graph.  Note that $\rho(G)$ cannot be isomorphic to $K_2$ and that
$\lambda_L(G)$ has at least one vertex if $G$ is not isomorphic to $K_2$.

\begin{theorem}\label{thm:confluence}
  Let $G$ be a connected graph other than $K_2$ and let $U$ be a subset of its
  leaves.  Then there exists a set of leaves $V$ of
  $\rho\circ\lambda(G)$ such that
  \[
    \rho\circ\lambda\circ \rho\circ \lambda_U(G) %
    =\rho\circ\lambda_V\circ\rho\circ\lambda(G).
  \]
\end{theorem}
One can actually show that the reduction of a graph can also be
obtained by repeatedly cutting off any single leaf from a graph other
than $K_2$ and reducing the size of any group of siblings by one, in
any order.  However, we do not need this statement, so we content
ourselves with the above.

\Cref{thm:confluence} is a simple consequence of the following lemma.
\begin{lemma}\label{lem:confluence}
  Let $H$ be a connected graph other than $K_2$ and let $L$ be a subset of its
  leaves.  Then
  \[
    \rho\circ\lambda_L(H) %
    =\rho\circ\lambda_L\circ\rho(H).
  \]
\end{lemma}
\begin{proof}
  We partition the set of vertices of $\lambda_L(H)$ into maximal
  groups of siblings $\{S_i\mid i\in I\}$.  Each set $S_i$
  corresponds to a triple $(P_i, L_i, R_i)$ of (possibly empty) sets
  of vertices in $H$, where $L_i\subseteq L$ is the set of leaves in $L$ of $H$ adjacent to $S_i$, $P_i\subseteq S_i$ is
  the set of vertices adjacent to a leaf in $L_i$ in $H$ and
  $R_i = S_i\setminus P_i$.  These sets partition the set of vertices of $H$.

  Let us first show that $R_i$ is a maximal group of siblings of $H$
  for all $i\in I$.  Let $x\in R_i$ and let $y\neq x$ be a sibling in
  $H$.  In a graph other than $K_2$, siblings cannot be leaves, so
  $x$ and $y$ are both vertices in $\lambda_L(H)$.  Moreover, in a
  graph other than $K_2$, siblings cannot be adjacent to leaves.
  Thus, the map $\lambda_L$ preserves neighborhoods of siblings, and
  $x$ and $y$ are also siblings in $\lambda_L(H)$.  Since
  $S_i=P_i\cup R_i$ is a maximal group of siblings in $\lambda_L(H)$,
  $y\in S_i$.  Finally, $y\not\in P_i$ because it is a sibling of $x$
  in $H$.

  It follows that the vertices of $\rho(H)$ can be partitioned into
  pairs of sets $(L_i, S'_i)$, with
  \[
    S'_i =
    \begin{cases}
      P_i &\text{if $R_i=\emptyset$}\\
      P_i \cup \{r_i\} &\text{otherwise,}
    \end{cases}
  \]
  where $r_i$ is the vertex obtained by contracting the maximal group
  of siblings $R_i$.  Consequently, the sets $S'_i$ partition
  $\lambda_L\circ\rho(H)$.  
  
  We now show that each set $S'_i$ is a
  maximal group of siblings in this graph.  Consider the surjective
  map
  \begin{align*}
    \iota&: \lambda_L(H)\to\lambda_L\circ\rho(H)\\
    z&\mapsto
       \begin{cases}
         r_i&\text{if $z\in R_i$ for some $i$}\\
         z&\text{otherwise}.
       \end{cases}
  \end{align*}
  For any two vertices $x$ and $y$ in $\lambda_L(H)$ we have that $x$
  and $y$ are equal or adjacent if and only if $\iota(x)$ and
  $\iota(y)$ are equal or adjacent in $\lambda_L\circ\rho(H)$.  Put
  differently,
  \[
    \cNhd_{\lambda_L(H)}(x) = \cNhd_{\lambda_L(H)}(y) %
    \Leftrightarrow %
    \cNhd_{\lambda_L\circ\rho(H)}(\iota(x)) =
    \cNhd_{\lambda_L\circ\rho(H)}(\iota(y)).
  \]

  Since each set $S'_i$ is a maximal group of siblings, it is contracted to a single vertex $s'_i$ when applying $\rho$ to $\lambda_L\circ\rho(H)$.  On the other hand, applying $\rho$ to $\lambda_L(H)$ contracts each set $S_i$ to a single vertex
  $s_i$ by definition. Finally, $s'_i$ and $s'_j$ are adjacent if and only if
  $s_i$ and $s_j$ are adjacent for any pair $i,j\in I$, because $\rho$ preserves adjacency between the groups of siblings it contracts.
\end{proof}

\begin{proof}[Proof of \Cref{thm:confluence}]
  If $\lambda_U(G)$ is isomorphic to $K_2$ then $\lambda(G)$ is either isomorphic to $K_2$ or a single vertex graph.  Thus we choose $V=\emptyset$
  and find that both sides of the equation
  \begin{equation}
    \label{eq:1}
    \rho\circ\lambda\circ \rho\circ \lambda_U(G) %
    =\rho\circ\lambda_V\circ\rho\circ\lambda(G).
  \end{equation}
  correspond to the single vertex graph.

  Otherwise, let $U^c$ be the complement of $U$
  in the set of leaves of $G$.  Since this is a subset of leaves of $\lambda_U(G)$, and $\lambda_U(G)$ is not isomorphic to $K_2$, we can apply
  \Cref{lem:confluence} with $H=\lambda_U(G)$ and obtain
  \begin{equation}\label{eq:2}
    \rho\circ\lambda(G) %
    = \rho\circ\lambda_{U^c}\circ\lambda_U(G)
    = \rho\circ\lambda_{U^c}\circ\rho\circ\lambda_U(G).
  \end{equation}

  If $\lambda_{U^c}\circ\rho\circ\lambda_U(G)$ is isomorphic to
  $K_2$, it follows that $\rho\circ\lambda(G)$ is the single vertex
  graph, and that $\lambda\circ\rho\circ\lambda_U(G)$ is either isomorphic to $K_2$
  or the single vertex graph.  We choose $V=\emptyset$
  and find that both sides of
  \Cref{eq:1} yield the single vertex graph.

  Finally, suppose that neither $\lambda_U(G)$ nor
  $\lambda_{U^c}\circ\rho\circ\lambda_U(G)$ are isomorphic to $K_2$.  Since siblings in a graph other than $K_2$ cannot be leaves and cannot be adjacent to leaves, $U^c$ is a subset of the set of leaves of $\rho\circ\lambda_U(G)$.  Let $V$ be the complement
  of $U^c$ in the set of leaves of $\rho\circ\lambda(G)$.

  Applying \Cref{lem:confluence} and taking into account that $V\cup U^c$ is the set of leaves of $\rho\circ\lambda_U(G)$, we have
  \[
    \rho\circ\lambda_V\circ\rho
    \circ\lambda_{U^c}\circ\rho\circ\lambda_U(G) %
    = \rho\circ\lambda_V\circ\lambda_{U^c}\circ\rho\circ\lambda_U(G)
    = \rho\circ\lambda\circ\rho\circ\lambda_U(G).
  \]
  Combining this with \Cref{eq:2}, the theorem follows.
\end{proof}

\begin{corollary}\label{cor:confluence}
  Let $U$ be a subset of leaves of a connected graph $G\neq K_2$.  Then
  $\rho\circ\lambda_U(G)$ and $G$ have the same reduction.
\end{corollary}
\begin{proof}
  Recall that each step in the reduction of $G$ consists of applying $\rho\circ \lambda$. We induct on the minimal integer $n$ such that
  $(\rho\circ\lambda)^n(G)$ is reduced.  The base case $n=0$, in
  which $U$ is necessarily empty, is trivial.

  Suppose that $(\rho\circ\lambda)^{n+1}(G)$ is reduced, with $n$
  minimal. Let $G' = \rho\circ\lambda(G)$. In particular $R(G) = R(G')$. By induction,
  $\rho\circ\lambda_V(G')$ and $G'$ have the same reduction for any
  set of leaves~$V$ of~$G'$.

  By \Cref{thm:confluence}, there exists a set of leaves $V$ of
  $G'=\rho\circ\lambda(G)$ such that
  \[
    \rho\circ\lambda\circ\rho\circ\lambda_U(G) =
    \rho\circ\lambda_V(G').
  \]
  Thus, denoting the reduction of a graph $G$ with $R(G)$, we have
  \[
    R\big(\rho\circ\lambda_U(G)\big)=R\big(\rho\circ\lambda\circ\rho\circ\lambda_U(G)\big) %
    = R\big(\rho\circ\lambda_V(G')\big) = R(G') = R(G). \qedhere\]
\end{proof}

\section{Refined Symmetry of the Sibling and the Tuft number}
\label{sec:symmetry-refined}

In this section, we prove \Cref{thm:R} by further refining the Patch Decomposition from
\Cref{thm:patch-decomposition} and \Cref{thm:patch-decomposition-siblings}.

Let $R$ be a connected leafless point distinguishing graph.
Let $\decorated^{R, \sibling}$, $\decorated^{R, \sibling, \tuft}$ and
$\comating^{R}$ be the subspecies of $\decorated^{\connected, \sibling}$,
$\decorated^{\connected, \sibling, \tuft}$ and $\comating^\connected$ respectively, whose
structures, regarded as graphs, reduce to a graph isomorphic to~$R$.

\begin{corollary}\label{thm:patch-decomposition-siblings-refined}
  For the single vertex graph $\bullet,$ we have
  \begin{align*}
    \decorated^{\bullet,\sibling}(X, X+Y, Z) %
    &= \comating^\bullet\circ\Patch^{\sibling} %
      - XY - [\sibling > 0] \Sets_2(Z),
      \intertext{whereas, for $R\neq\bullet$, we have}\notag
      \decorated^{R,\sibling}(X, X+Y, Z) &= \comating^R\circ\Patch^{\sibling}.
  \end{align*}
\end{corollary}
\begin{proof}
  \sloppypar
  By the Patch Decomposition from 
  \Cref{thm:patch-decomposition} we have an equivariant bijection
  between graphs $G\in\decorated^\connected(X, X+Y, Z)[U]$ and triples
  $\left(\pi, D, \left(P_B\right)_{B\in\pi}\right)$, where $D$ is the underlying point distinguishing graph of~$G$. By \Cref{thm:patch-decomposition-siblings}, this bijection preserves the
  sibling number.
  \Cref{cor:confluence} applied to the set of leaves of sort $Y$ of $G$ implies that~$G$ and its underlying point distinguishing graph~$D$ reduce to the same graph $R$.

  If $R$ is not the single vertex graph, $G$ has at least~$4$
  vertices, so the exceptional cases corresponding to $XY$ and $\Sets_2(Z)$ depicted in \Cref{fig:exception} do not occur.
\end{proof}

\begin{theorem}\label{thm:enumeration-refined}
  For the single vertex graph $\bullet$, we have
  \begin{align*}
    \decorated^{\bullet,\sibling,\tuft}(X, Y, Z) %
    &= \comating^\bullet\circ\mathcal Q^{\sibling,\tuft} %
    + X^2 - [\tuft > 0] X Y - [\sibling > 0] \Sets_2(Z),\\
      \intertext{whereas, if $R\neq\bullet$, we have}
    \decorated^{R,\sibling,\tuft}(X, Y, Z) %
    &= \comating^R\circ\mathcal Q^{\sibling,\tuft}.
  \end{align*}
\end{theorem}
\begin{proof}
  We apply \Cref{lem:decoratedgeneric} in the case where 
  $\decoratedgeneric = \decorated^{R, \sibling}$ is the species of
  connected graphs other than $K_2$ with sibling number at most $\sibling$ which reduce to $R$.  The closure property holds because replacing a positive number of leaves adjacent to a vertex by a positive number of leaves neither affects the reduction of the graph nor its sibling number, unless the resulting graph is a $K_2$.
  
  The star graphs with at
  least three vertices have sibling number zero and reduce to a single vertex.  Thus, the first equation in \Cref{lem:decoratedgeneric} applies if and only if $R$ is the single vertex graph.
  After setting
  $Y = \Omega(\EInProofOfThm_\tuft) - X$ in
  \Cref{thm:patch-decomposition-siblings-refined}, the claim follows from
  \Cref{lem:Pst}.
\end{proof}

\begin{proof}[Proof of \Cref{thm:R}]
  This follows from \Cref{thm:enumeration-refined} in the same way as
  \Cref{thm:main} follows from \Cref{thm:enumeration}.
\end{proof}

\section{Induced Cycles and Graph Reduction}\label{sec:graph-classes}
An \Dfn{induced subgraph} of a graph $G$ consists of a subset $V$ of vertices of $G$ together with all edges of $G$ connecting vertices within $V$. An \Dfn{induced $k$-cycle} is an induced subgraph that is isomorphic to a cycle graph with $k$ vertices. In this section, we show that induced cycles with at least~$4$ vertices
are essentially preserved by reduction.
\begin{theorem}\label{thm:inducedcycles}
  For $k\geq 4$, a graph $G$ contains an induced $k$-cycle if and
  only if the reduction of $G$ contains an induced $k$-cycle.

  In particular, the sibling number and the tuft number have joint
  symmetric distribution on the set of connected chordal graphs,
  i.e., on graphs which do not contain an induced cycle with~$4$ or more
  vertices.
\end{theorem}
\begin{proof}
  We show that induced cycles are preserved when removing leaves and
  when contracting any maximal group of siblings to a single vertex.
  The former is true, because an induced cycle of a graph cannot
  contain a leaf.

  To prove the latter, we first show that a group of siblings~$S$ and
  an induced $k$-cycle~$C$ with $k\geq 4$ can have at most one vertex
  in common.  Let~$v$ be such a vertex and let~$u$ and~$w$ be its
  neighbors on the induced $k$-cycle.  Because $k\geq 4$ the
  vertices~$u$ and~$w$ cannot be adjacent.  Thus, neither of them can
  be a sibling of~$v$.

  Let $G'$ be the graph obtained by contracting the vertices in a
  maximal group of siblings $S$ to a single vertex.  If $S$ and $C$
  have no vertices in common, $C$ is also an induced cycle of $G'$.
  If $S$ and $C$ have a single vertex $v$ in common, contracting the
  group of siblings $S$ to a single vertex $v'$, the vertices of
  $(C\setminus\{v\})\cup\{v'\}$ form an induced cycle of $G'$.

  Conversely, $\lambda(G)$ is an induced subgraph of $G$, and
  $\rho(G)$ is isomorphic to the subgraph induced on one representative
  of each maximal group of siblings.  Thus, the reduction of $G$ is
  isomorphic to an induced subgraph of~$G$.
\end{proof}

\section{Point Distinguishing Graphs Reducing to a Given Graph}\label{sec:Description of Comating}

In this section, we show that the species $\comating^R$ of
point distinguishing graphs which reduce to a given connected leafless point distinguishing graph $R$ can be obtained from the species $\comating^\bullet$ of
point distinguishing graphs which reduce to the single vertex graph with little
computational effort.  Additionally, we provide the first few terms of the molecular
decompositions of the species $\comating^\bullet$ and the species of
leafless point distinguishing graphs $\decorated^\connected(X,0,0)$.

Recall from \Cref{ex:inverses}~(3) that the compositional inverse of $X-X^2$ is the species of ordered rooted trees $\mathcal A$.
\begin{proposition}\label{prop:comating}
  Let $R$ be a connected leafless point distinguishing graph and let $\mathcal R$ be the species whose unique isomorphism type
  is $R$. Then
  \begin{align*}
    \comating^\bullet &= (X-X^2)\circ\Patch(X, -X, 0)^{(-1)}\\
    \intertext{whereas, if $R\neq\bullet$,}
    \comating^R &= \mathcal R\circ\Patch(X, -X, 0)^{(-1)}. %
  \end{align*}
  In particular, if $R\neq\bullet$, we have that 
  \[
  \comating^R = \mathcal R\circ\mathcal A\circ\comating^\bullet.
  \]
\end{proposition}
It is remarkable that the compositional inverse $\Patch(X, -X, 0)^{(-1)} = \mathcal A\circ\comating^\bullet$  of the virtual species $P(X, -X, 0)=(1+X)\cdot\Sets\big(X\Sets_{\geq1}(-X)\big)-1$ is a proper (that is, non-virtual) species.
\begin{proof}
  For any connected leafless point distinguishing graph $R$ we have
  $\decorated^{R, s}(X,0,0) = \mathcal R$, because the only leafless point distinguishing graph reducing to $R$ is $R$ itself.
  Thus, setting $Y=-X$ and $Z=0$ in
  \Cref{thm:patch-decomposition-siblings-refined}, we obtain
  \[
    X=\comating^\bullet\circ\Patch^\sibling(X,-X,0)+X^2,
  \]
  if $R$ is the single vertex graph, and
  \[
    \mathcal R = \comating^R \circ\Patch^\sibling(X, -X, 0),
  \]
  otherwise.  Since $\Patch^{\sibling}(X, -X, 0) = \Patch(X, -X, 0)$
  the claim follows.
\end{proof}

Using these identities and the \SageMath\ implementation of species,\footnote{developed by Mainak
  Roy and the third author, included since version 10.7

  \url{https://doc.sagemath.org/html/en/reference/combinat/sage/rings/lazy_species.html}} it is easy to compute the first few terms of
the molecular decompositions and isomorphism type generating series
of the species $\comating^\bullet$ and $\Patch(X, -X, 0)^{(-1)}$.
\VerbatimFootnotes
The molecular decomposition of the species $\comating^\bullet$ begins
with\footnote{\verb|L.<X> = LazyCombinatorialSpecies(QQ)|

\verb|E = L.Sets(); P = E(X*E(-X)-X)*(1+X)-1; (X-X^2)(P.revert())|}
\begin{align*}
  X & + X \Sets_2 + X \Sets_3+\Sets_2(X^2) + X \Sets_4+2 X^3 \Sets_2+2 X \Sets_2(X^2) \\
    & + X \Sets_5+\Sets_2(X \Sets_2)+X^2 \Sets_2^2+\Sets_2 \Sets_2(X^2)+3 X^4 \Sets_2 \\
    & + \Sets_2(X^3)+\Sets_3(X^2)+X^2 \Sets_2(X^2)+2 X^3 \Sets_3+2 X^6 + \dots
\end{align*}
The corresponding isomorphism type generating series begins with
\[
  x + x^3 + 2 x^4 + 5 x^5 + 14 x^6 + 43 x^7 + 141 x^8 + 491 x^9 + 1778 x^{10} + \dots
\]
The molecular decomposition of $\Patch(X, -X, 0)^{(-1)}$ begins with
\begin{align*}
  X &+ X^2 + X \Sets_2+2 X^3 %
      + X \Sets_3+2 X^2 \Sets_2+\Sets_2(X^2)+5 X^4 \\
    &+ X \Sets_4+2 X^2 \Sets_3+8 X^3 \Sets_2+4 X \Sets_2(X^2)+14 X^5 \\
    &+ X \Sets_5+2 X^2 \Sets_4+\Sets_2(X \Sets_2)+2 X^2 \Sets_2^2+8 X^3 \Sets_3+\Sets_2 \Sets_2(X^2)\\
    &+\,27 X^4 \Sets_2+\Sets_2(X^3)+\Sets_3(X^2)+11 X^2 \Sets_2(X^2)+44 X^6 + \dots
\end{align*}
The corresponding isomorphism type generating series begins with
\[
  x + x^2 + 3 x^3 + 9 x^4 + 29 x^5 + 99 x^6 + 353 x^7 + 1300 x^8 + 4913 x^9 + 18945 x^{10} + \dots
\]

The isomorphism types of the species $\decorated^\connected(X,0,0)$, the species
of graphs having no siblings or leaves, are precisely the graphs $R$
that arise as the reduction of some graph.  The molecular
decomposition of this species seems to be significantly harder to
compute.  A naive approach, enumerating all graphs, yields the
following.
\begin{align*}
  X &+ P_4 + \Sets_2 \Sets_3+P_5+X P_4+2 X \Sets_2(X^2)\\
    &+\Sets_2 \Sets_4+\Sets_3(\Sets_2)+3 \Sets_2 \Sets_2(X^2)+\Sets_2(\Sets_3)+4 X^4 \Sets_2
      +2 P_6+\Sets_2(X \Sets_2)\\
    & +2 \big(X^2\Sets_2(X^2)\big)/\mathbb Z_2 + \Sets_3(X^2)+X \Sets_2 \Sets_3 +X^2 \Sets_2^2+X P_5+8 X^2 \Sets_2(X^2)\\
    & +2 \Sets_2(X^3) +2 X^6 + \dots
\end{align*}
Here, $P_n$ denotes the species of polygons of size $n$,
corresponding to a graph whose automorphism group is the dihedral
group of degree $n$ and order $2n$.  Moreover, $\big(X^2\Sets_2(X^2)\big)/\mathbb Z_2$ is the species of polygons on six vertices with a single diameter which connects two opposite vertices.
The corresponding isomorphism type generating series begins with
\[
  x + x^4 + 5 x^5 + 31 x^6 + 293 x^7 + 4986 x^8 + 151096 x^9 + 8264613 x^{10} + \dots
\]
By \Cref{cor:bipointdetermining} this generating series coincides with the generating series of bi-point-determining graphs.\footnote{\OEIS{A129586}} \Cref{fig:Graphs5} shows the connected leafless point distinguishing graphs with five
vertices corresponding to the five terms of size $5$ in the molecular decomposition of $\decorated^\connected(X,0,0)$.

\begin{figure}
  \centering
  \begin{tikzpicture}
    \node[vertex] (V1) at (0,0) {};
    \node[vertex] (V2) at (1.41421356237310,0) {};
    \node[vertex] (V3) at (1.41421356237310,1.41421356237310) {};
    \node[vertex] (V4) at (0,1.41421356237310) {};
    \node[vertex] (V5) at (0.707106781186548,0.707106781186548) {};
    \node at (0.707106781186548,-1.292893218813452) {$\Sets_2\Sets_3$};
    \draw (V1)--(V2)--(V3)--(V4)--(V1);
    \draw (V1)--(V5)--(V3);
  \end{tikzpicture}
  \begin{tikzpicture}
    \node[vertex] (V1) at (-0.587785252292473, -0.809016994374947) {};
    \node[vertex] (V2) at (0.587785252292473,  -0.809016994374947) {};
    \node[vertex] (V3) at (0.951056516295154,   0.309016994374947) {};
    \node[vertex] (V4) at (0.000000000000000,   1.00000000000000)  {};
    \node[vertex] (V5) at (-0.951056516295154,  0.309016994374947) {};
    \node at (0,-2) {$P_5$};
    \draw (V1)--(V2)--(V3)--(V4)--(V5)--(V1);
  \end{tikzpicture}
  \begin{tikzpicture}
    \node[vertex] (V1) at (0,0) {};
    \node[vertex] (V2) at (1.41421356237310,0) {};
    \node[vertex] (V3) at (1.41421356237310,1.41421356237310) {};
    \node[vertex] (V4) at (0,1.41421356237310) {};
    \node[vertex] (V5) at (0.707106781186548,0.707106781186548) {};
    \node at (0.707106781186548,-1.292893218813452) {$XP_4$};
    \draw (V1)--(V2)--(V3)--(V4)--(V1)--(V5)--(V3);
    \draw (V2)--(V5)--(V4);
  \end{tikzpicture}
  \begin{tikzpicture}
    \node[vertex] (V1) at (-0.587785252292473, -0.809016994374947) {};
    \node[vertex] (V2) at (0.587785252292473,  -0.809016994374947) {};
    \node[vertex] (V3) at (0.951056516295154,   0.309016994374947) {};
    \node[vertex] (V4) at (0.000000000000000,   1.00000000000000)  {};
    \node[vertex] (V5) at (-0.951056516295154,  0.309016994374947) {};
    \node at (0,-2) {$X\Sets_2(X^2)$};
    \draw (V1)--(V2)--(V3)--(V4)--(V5)--(V1);
    \draw (V5)--(V3);
  \end{tikzpicture}
  \begin{tikzpicture}
    \node[vertex] (V1) at (-0.587785252292473, -0.809016994374947) {};
    \node[vertex] (V2) at (0.587785252292473,  -0.809016994374947) {};
    \node[vertex] (V3) at (0.951056516295154,   0.309016994374947) {};
    \node[vertex] (V4) at (0.000000000000000,   1.00000000000000)  {};
    \node[vertex] (V5) at (-0.951056516295154,  0.309016994374947) {};
    \node at (0,-2) {$X\Sets_2(X^2)$};
    \draw (V1)--(V2)--(V3)--(V4)--(V5)--(V1);
    \draw (V5)--(V3);
    \draw (V5)--(V2);
  \end{tikzpicture}
  \caption{All unlabeled connected leafless point distinguishing graphs with $5$ vertices and their corresponding term in the molecular decomposition of $\decorated^\connected(X,0,0)$.}
  \label{fig:Graphs5}
\end{figure}

\section{Towards a Bijection}\label{sec:bijection}
To prove \Cref{thm:main} and \Cref{thm:R} bijectively, one might hope
to construct an involution on the set of unlabeled connected graphs with a given number
of vertices, which interchanges groups of siblings with tufts and
preserves the reduction.  In this section, we will discuss this idea
and its pitfalls.

To an unlabeled connected graph $G$ other than $K_2$ we associate the (unlabeled) \Dfn{decorated
  graph} $G_\dec$ as follows:
\begin{itemize}
\item we tag each vertex of $G$ that is adjacent to $k$ leaves with $T_k$, and remove all leaves from $G$,
\item we replace each maximal group of siblings of size $k+1$ of $G$
  by a single vertex, which we tag with $S_k$.
\end{itemize}

An example for this process is displayed in \Cref{fig:Figure
  decoration of graphs}. Since a vertex adjacent to a leaf cannot
have siblings, $G_\dec$ is well-defined and each of its vertices has
at most one tag.\footnote{a very similar idea appears in \cite{BurchardtHahn}}  The sibling number $\sibling(G)$ is the maximal
number $s$ such that $S_s$ appears as a tag, and the tuft number
$\tuft(G)$ is the maximal number $t$, such that $T_t$ appears as a
tag.
\begin{figure}[ht]
\centering
  \begin{tikzpicture}[baseline={(0,0)}]
     \node[vertex2] (V10) at (-0.525,0.05) {};
     \node[vertex2] (V11) at (0.05,-0.525) {};
     \node[vertex2] (V12) at (0.26,0.26) {};

      \node[vertex2] (V21) at (2,0) {};
      \node[vertex2] (V22) at (0,2) {};

      \node[vertex2] (V3) at (1,-1) {};
      \node[vertex2] (V4) at (-1,1) {};

      \node[vertex2] (V5) at (-1,-1) {};
      \node[vertex2] (V51) at (-1.8,-.2) {};
      \node[vertex2] (V52) at (-1.8,-1) {};
      \node[vertex2] (V53) at (-1.8,-1.8) {};

      \node[vertex2] (V6) at (2,2) {};

      \node[vertex2] (V71) at (-.4,-2) {};
      \node[vertex2] (V712) at (-.8,-2.8) {};
      \node[vertex2] (V713) at (0,-2.8) {};

      \node[vertex2] (V72) at (.4,-2) {};

      \node[vertex2] (V81) at (-2.4,1.6) {};
      \node[vertex2] (V811) at (-2.4,.8) {};
      \node[vertex2] (V82) at (-1.6,2.4) {};

      \draw (V10)--(V11)--(V12)--(V10) (V10)--(V21)--(V3) (V10)--(V22)--(V3) (V4)--(V5)--(V10) (V11)--(V21) (V21)--(V22) (V11)--(V22) (V5)--(V11) (V12)--(V21) (V12)--(V22) (V5)--(V12) (V3)--(V5) (V21)--(V4) (V21)--(V6) (V22)--(V4) (V22)--(V6) (V3)--(V71) (V4)--(V81) (V4)--(V82) (V81)--(V82) (V71)--(V5) (V3)--(V72)--(V71) (V72)--(V5) (V51)--(V5)--(V52) (V71)--(V712) (V712)--(V713)--(V71) (V811)--(V81) (V53)--(V5);
  \end{tikzpicture}\hspace{1cm}
     \begin{tikzpicture}[baseline={(0,0)}]
     \node[vertex2] (V1) at (-0.07,-0.07) {$\scriptstyle S_2$};

      \node[vertex2] (V2) at (1,1) {$\scriptstyle S_1$};

      \node[vertex2] (V3) at (1,-1) {};
      \node[vertex2] (V4) at (-1,1) {};

      \node[vertex2] (V5) at (-1,-1) {$\scriptstyle T_3$};

      \node[vertex2] (V6) at (2,2) {};

      \node[vertex2] (V71) at (-.4,-2) {};
      \node[vertex2] (V711) at (-.4,-2.8) {$\scriptstyle S_1$};

      \node[vertex2] (V72) at (.4,-2) {};

      \node[vertex2] (V81) at (-2.4,1.6) {$\scriptstyle T_1$};
      \node[vertex2] (V82) at (-1.6,2.4) {};

      \draw (V1)--(V2)--(V3) (V4)--(V5)--(V1) (V3)--(V5) (V2)--(V4) (V2)--(V6) (V3)--(V71)--(V72)--(V3) (V4)--(V81)--(V82)--(V4) (V711)--(V71)--(V5)--(V72);
  \end{tikzpicture}\hspace{.5cm}
  \caption{An unlabeled graph and its associated decorated graph.}
  \centering
  \label{fig:Figure decoration of graphs}
\end{figure}

Let a \Dfn{tagged graph} be an unlabeled graph each of whose vertices may be
untagged, have a tag $T_k$ or a tag $S_k$ for some positive
integer~$k$.  We can, in fact, characterize the tagged graphs that
occur as decorated graphs.
\begin{proposition}\label{lem:decorated}
  Given an unlabeled connected graph $G$ other than $K_2$, its decorated graph $G_\dec$ satisfies
  the following conditions:
  \begin{enumerate}
  \item An isolated vertex is not tagged $S_1$ or $T_1$.
  \item If two vertices of $G_\dec$ are siblings, at least one of
    them is tagged $T_k$, for some~$k$.
  \item Every untagged leaf is a neighbor of a vertex with
    tag~$S_k$, for some~$k$.
  \end{enumerate}
  Conversely, any connected tagged graph satisfying these conditions is the
  decorated graph of an unlabeled connected graph other than $K_2$.
\end{proposition}

\begin{proof}
  The first condition accounts for the fact that we exclude $K_2$.
  
  To show that the second condition holds, let~$u$ and~$v$ be two
  siblings of $G_\dec$, such that neither of them has a tag
  $T_k$ for any~$k$.  By construction, the preimages of~$u$ and~$v$
  in~$G$ cannot be siblings, that is they have distinct closed neighborhoods.  Let $w$ be any vertex of $G$ which is, without loss of generality, in the open
  neighborhood of the preimage of~$u$ but not in the open neighborhood of the preimage of~$v$.  The vertex $w$ cannot be a leaf because~$u$ has no tag~$T_k$ in $G_\dec$. Thus, its image in $G_\dec$, which is distinct from~$u$ and~$v$, is in the neighborhood of~$u$ but not in
  the neighborhood of~$v$, contradicting the assumption that $u$ and $v$ are siblings.

  For the third condition, let~$u$ be an untagged vertex in
  $G_\dec$. Then, by construction, its preimage cannot be a leaf
  of~$G$.  Thus, if $u$ is a leaf in $G_\dec$, its preimage must have multiple
  neighbors in~$G$ which necessarily are siblings.

  Conversely, suppose that a tagged graph satisfies all three conditions.  Replacing the tags
  with their corresponding structures and then constructing its decorated graph,
  results in the original tagged graph, because vertices which are leaves or adjacent to leaves cannot have siblings.
\end{proof}

One might hope that changing each tag $S_n$ to a tag $T_n$ and vice
versa provides a bijection interchanging the sibling and the tuft
number.  However, the validity of the conditions given in
\Cref{lem:decorated} is not preserved by this operation.  For
example, the tagged graph on the left in \Cref{fig:counterexample ideal bij}
is a valid decorated graph, whereas the tagged graph on the right violates
the third condition.
\begin{figure}[ht]
  \centering
  \begin{tabular}{c@{\quad$\leftrightarrow$\quad}c}
    \begin{tikzpicture}[baseline={(current bounding box.center)}]
      \node[vertex] (V1) at (0,1) {};
      \node[vertex] (V2) at (1,0) {$\scriptstyle S_n$};
      \node[vertex] (V3) at (2,1) {};
      \draw (V1)--(V2)--(V3);
    \end{tikzpicture}
    &
      \begin{tikzpicture}[baseline={(current bounding box.center)}]
        \node[vertex] (V1) at (0,1) {};
        \node[vertex] (V2) at (1,0) {$\scriptstyle T_n$};
        \node[vertex] (V3) at (2,1) {};
        \draw (V1)--(V2)--(V3);
      \end{tikzpicture}
  \end{tabular}
  \caption{A tagged graph satisfying the conditions of
    \Cref{lem:decorated} on the left and an invalid tagged graph on the
    right.}
  \label{fig:counterexample ideal bij}
\end{figure}

Because of \Cref{thm:R}, we want to find an involution $\iota$ on the set of unlabeled connected graphs that
interchanges the sibling and the tuft number, and is equivariant with respect to reduction.
Additionally, we may require that $\iota$ replaces any tag $S_k$ with
a tag $T_k$ and conversely, on the corresponding decorated graph,
whenever the result is a valid decorated graph.  Under these
assumptions, \Cref{fig:bijection} displays pairs of sets of decorated
graphs that are necessarily in bijection.  Note that these examples
demonstrate that a bijection cannot preserve the number of tags.
\begin{figure}[h]
  \centering
  \renewcommand{\arraystretch}{3}
  \begin{tabular}{c@{\quad$\leftrightarrow$\quad}c}
    \begin{tikzpicture}[scale=0.8, baseline={(current bounding box.center)}]
      \node[vertex] (V1) at (0,1) {};
      \node[vertex] (V2) at (1,0) {$\scriptstyle S_n$};
      \node[vertex] (V3) at (2,1) {};
      \draw (V1)--(V2)--(V3);
    \end{tikzpicture}
    &
    \begin{tikzpicture}[scale=0.8, baseline={(current bounding box.center)}]
      \node[vertex] (V1) at (0,0) {$\scriptstyle T_n$};
      \node[vertex] (V2) at (1,0) {$\scriptstyle T_1$};
      \draw (V1)--(V2);
    \end{tikzpicture}
    \\
    \begin{tikzpicture}[scale=0.8, baseline={(current bounding box.center)}]
      \node[vertex] (V1) at (0,1) {$\scriptstyle T_k$};
      \node[vertex] (V2) at (1,0) {$\scriptstyle S_n$};
      \node[vertex] (V3) at (2,1) {};
      \draw (V1)--(V2)--(V3);
    \end{tikzpicture}
    &
    \begin{tikzpicture}[scale=0.8, baseline={(current bounding box.center)}]
      \node[vertex] (V1) at (0,1) {$\scriptstyle T_n$};
      \node[vertex] (V2) at (1,0) {$\scriptstyle S_k$};
      \node[vertex] (V3) at (2,1) {};
      \draw (V1)--(V2)--(V3);
    \end{tikzpicture}
    \\
    \begin{tikzpicture}[scale=0.8, baseline={(current bounding box.center)}]
      \node[vertex] (V1) at (1,0) {$\scriptstyle S_1$};
      \node[vertex] (V2) at (0,1) {};
      \node[vertex] (V3) at (1,1) {};
      \node[vertex] (V4) at (2,1) {};
      \draw (V1)--(V2);
      \draw (V1)--(V3);
      \draw (V1)--(V4);
    \end{tikzpicture}
    \;
    \begin{tikzpicture}[scale=0.8, baseline={(current bounding box.center)}]
      \node[vertex] (V1) at (0,0) {$\scriptstyle S_1$};
      \node[vertex] (V2) at (1,0) {$\scriptstyle S_1$};
      \node[vertex] (V3) at (1,1) {};
      \draw (V1)--(V2)--(V3);
    \end{tikzpicture}
    &
    \begin{tikzpicture}[scale=0.8, baseline={(current bounding box.center)}]
      \node[vertex] (V1) at (0,0) {$\scriptstyle T_1$};
      \node[vertex] (V2) at (1,0) {};
      \node[vertex] (V3) at (0,1) {};
      \node[vertex] (V4) at (1,1) {};
      \draw (V1)--(V2);
      \draw (V1)--(V3);
      \draw (V1)--(V4);
      \draw (V2)--(V3);
      \draw (V2)--(V4);
    \end{tikzpicture}
    \;
    \begin{tikzpicture}[scale=0.8, baseline={(current bounding box.center)}]
      \node[vertex] (V1) at (0,0) {$\scriptstyle T_1$};
      \node[vertex] (V2) at (1,0) {$\scriptstyle T_1$};
      \node[vertex] (V3) at (0.5,1) {};
      \draw (V1)--(V2)--(V3)--(V1);
    \end{tikzpicture}
    \\
    \begin{tikzpicture}[scale=0.8, baseline={(current bounding box.center)}]
      \node[vertex] (V1) at (1,0) {$\scriptstyle S_n$};
      \node[vertex] (V2) at (0,1) {};
      \node[vertex] (V3) at (1,1) {};
      \node[vertex] (V4) at (2,1) {};
      \draw (V1)--(V2);
      \draw (V1)--(V3);
      \draw (V1)--(V4);
    \end{tikzpicture}
    \;
    \begin{tikzpicture}[scale=0.8, baseline={(current bounding box.center)}]
      \node[vertex] (V1) at (0,0) {$\scriptstyle S_1$};
      \node[vertex] (V2) at (1,0) {$\scriptstyle S_n$};
      \node[vertex] (V3) at (1,1) {};
      \draw (V1)--(V2)--(V3);
    \end{tikzpicture}
    \;
    \begin{tikzpicture}[scale=0.8, baseline={(current bounding box.center)}]
      \node[vertex] (V1) at (0,0) {$\scriptstyle S_n$};
      \node[vertex] (V2) at (1,0) {$\scriptstyle S_1$};
      \node[vertex] (V3) at (1,1) {};
      \draw (V1)--(V2)--(V3);
    \end{tikzpicture}
    &
    \begin{tikzpicture}[scale=0.8, baseline={(current bounding box.center)}]
      \node[vertex] (V1) at (0,0) {$\scriptstyle T_n$};
      \node[vertex] (V2) at (1,0) {};
      \node[vertex] (V3) at (0,1) {};
      \node[vertex] (V4) at (1,1) {};
      \draw (V1)--(V2);
      \draw (V1)--(V3);
      \draw (V1)--(V4);
      \draw (V2)--(V3);
      \draw (V2)--(V4);
    \end{tikzpicture}
      \;
    \begin{tikzpicture}[scale=0.8, baseline={(current bounding box.center)}]
      \node[vertex] (V1) at (0,0) {$\scriptstyle T_n$};
      \node[vertex] (V2) at (1,0) {$\scriptstyle T_1$};
      \node[vertex] (V3) at (0.5,1) {};
      \draw (V1)--(V2)--(V3)--(V1);
    \end{tikzpicture}
      \;
    \begin{tikzpicture}[scale=0.8, baseline={(current bounding box.center)}]
      \node[vertex] (V1) at (0,0) {$\scriptstyle T_n$};
      \node[vertex] (V2) at (1,0) {$\scriptstyle T_2$};
      \draw (V1)--(V2);
    \end{tikzpicture}
  \end{tabular}
  \caption{Some sets of decorated graphs that reduce to a single vertex
    necessarily corresponding to each other by \Cref{thm:R}.  The last set requires $n>1$.}
  \label{fig:bijection}
\end{figure}

\subsection*{Acknowledgements}
We are grateful to Mainak Roy, who
provided a generic implementation of the molecular decomposition of
species, and to Travis Scrimshaw, who laid the foundations for lazy
computations with graded algebras in \SageMath.

We are very grateful to the anonymous referees for very constructive remarks that significantly improved the exposition.

\sloppy
The first named author was funded by the Austrian Science Fund (FWF), grant \href{https://doi.org/10.55776/P34765}{10.55776/P34765} and a DOC Fellowship (27150) of the Austrian Academy of Sciences. Moreover, he was supported by the French–Austrian project EAGLES (ANR-22-CE91-0007 \& FWF grant \href{https://doi.org/10.55776/I6130}{10.55776/I6130}).
The second author acknowledges the financial support from FWF, grant \href{https://doi.org/10.55776/P34931}{10.55776/P34931}.

\printbibliography
\end{document}